\documentclass[11pt,reqno,twoside]{amsart}
\usepackage{amssymb,amsmath,amsthm,soul,color,paralist}
\usepackage{t1enc}
\usepackage{comment}
\usepackage[cp1250]{inputenc}
\usepackage{a4,indentfirst,latexsym}
\usepackage{graphics}
\usepackage{mathrsfs}
\usepackage{cite,enumitem,graphicx}
\usepackage[colorlinks=true,urlcolor=blue,
citecolor=red,linkcolor=blue,linktocpage,pdfpagelabels,
bookmarksnumbered,bookmarksopen]{hyperref}
\usepackage[english]{babel}
\usepackage[left=2.30cm,right=2.30cm,top=2.72cm,bottom=2.72cm]{geometry}
\usepackage[metapost]{mfpic}
\usepackage[colorinlistoftodos]{todonotes}
\usepackage[normalem]{ulem}

\allowdisplaybreaks 

\makeatletter
\providecommand\@dotsep{5}
\def\listtodoname{List of Todos}
\def\listoftodos{\@starttoc{tdo}\listtodoname}
\makeatother

\numberwithin{equation}{section}

\def\Xint#1{\mathchoice 
  {\XXint\displaystyle\textstyle{#1}}%
  {\XXint\textstyle\scriptstyle{#1}}%
  {\XXint\scriptstyle\scriptscriptstyle{#1}}%
  {\XXint\scriptscriptstyle\scriptscriptstyle{#1}}%
  \!\int} 
\def\XXint#1#2#3{{\setbox0=\hbox{$#1{#2#3}{\int}$} 
  \vcenter{\hbox{$#2#3$}}\kern-.5\wd0}} 
\def\-int{\Xint -}

\newcommand{\R}{\mathbb{R}}
\newcommand{\N}{\mathbb{N}}
\newcommand{\ri}{\rightarrow}
\newcommand{\dis}{\displaystyle}

\DeclareMathOperator{\dive}{div}

\DeclareMathOperator{\B}{\mathcal{B}}

\DeclareMathOperator{\C}{\mathcal{C}}
\DeclareMathOperator{\A}{\mathcal{A}}
\DeclareMathOperator{\h}{\mathcal{H}}
\DeclareMathOperator{\w}{\rm{w}}
\DeclareMathOperator{\V}{\mathcal{\varphi}}
\DeclareMathOperator{\e}{\varepsilon}
\DeclareMathOperator{\X}{\mathbb{X}}
\DeclareMathOperator{\Y}{\mathbb{Y}}
\DeclareMathOperator{\W}{\mathbb{W}}
\DeclareMathOperator{\Z}{\mathbb{Z}}

\newtheorem{lem}{Lemma}[section]
\newtheorem{thm}{Theorem}[section]
\newtheorem{defn}{Definition}[section]
\newtheorem{cor}{Corollary}[section]

\newtheorem{remark}{Remark}[section]
\newtheorem{assumption}[thm]{Assumption}

\begin{document}
\title[Partial regularity result for non-autonomous elliptic systems with general growth]{Partial regularity result for non-autonomous elliptic systems with general growth}

\author[Teresa Isernia]{Teresa Isernia}
\address{Teresa Isernia \hfill\break\indent
Dipartimento di Ingegneria Industriale e Scienze Matematiche \hfill\break\indent
Universit\`a Politecnica delle Marche\hfill\break\indent
via brecce bianche 12, 60131 Ancona, Italy}
\email{t.isernia@univpm.it}
\author{Chiara Leone}
\address{}
\author{Anna Verde}
\address{Chiara Leone and Anna Verde \hfill\break\indent
Dipartimento di Matematica e Applicazioni "R. Caccioppoli" \hfill\break\indent
         Universit\`a degli Studi di Napoli Federico II \hfill\break\indent
         via Cinthia, 80126 Napoli, Italy}
\email{chileone@unina.it, anverde@unina.it}



\begin{abstract}
In this paper we prove a H\"older partial regularity result for weak solutions $u:\Omega\to \mathbb{R}^N$, $N\geq 2$, to non-autonomous elliptic systems with general growth of the type:
\begin{equation*}
-\dive a(x, u, Du)= b(x, u, Du) \quad \mbox{ in } \Omega. 
\end{equation*}
The crucial point is that the operator $a$ satisfies very weak regularity properties and a general growth, while the inhomogeneity $b$ has a controllable growth. 
\end{abstract}

\maketitle

\section{Introduction}

\noindent
In this paper we are concerned with the partial H\"older continuity of weak solutions to elliptic systems with general growth. More precisely we consider
\begin{equation}\label{P}
-\dive a(x, u, Du)= b(x, u, Du) \quad \mbox{ in } \Omega, 
\end{equation}
where $\Omega \subset \R^{n}$ is a bounded open set of dimension $n\geq 2$, the structure function $a:\Omega \times \R^{N}\times \R^{nN}\ri \R^{nN}$ satisfies ellipticity and growth conditions in terms of Orlicz functions, while the inhomogeneity $b:\Omega \times \R^{N}\times \R^{nN}\ri \R^{N}$ fulfills a controllable growth condition (see Subsection \ref{hpoperatore}). 

\noindent The role of the ``coefficients'' $(x,u)$ of the function $a$ and the Orlicz growth will be the focus of this paper. We will suppose that the function $a$ satisfies a VMO-condition with respect to $x$ and a continuity with respect to $u$ (for the precise assumptions on $a$ we refer to Subsection \ref{hpoperatore}). In this regard, let us emphasize that partial H\"older continuity is the best one can expect.

\noindent Partial regularity of weak solutions to elliptic systems with $p$-growth has been extensively studied. In fact, in light of the pioneering works of Giusti-Miranda \cite{GM} and Morrey \cite{Morrey}, the question regarding conditions on the function $a$ in order to  obtain partial regularity has been for long investigated.
In the case of power growth, allowing H\"older continuity of $a$ with respect to the first two variables, it is well known that the gradients of weak solutions are H\"older continuous outside a set of measure zero (see \cite{DG} and the related references). Assuming only a continuity in $(x,u)$, in view of the established results in the scalar case, the natural expectation is that $u$ is partially H\"older continuous: this has been proven  in \cite{FM}. 
Later, these results were generalized to the case where the coefficients satisfy only a VMO-condition in $x$ (hence possibly discontinuous) combined with a continuity assumption with respect to the second variable (see \cite{BDHS, BJDE}). Actually many authors contributed with interesting results prior to these, including \cite{KM1,KM2, KM3, RT1,DES,Z}. 


\noindent The main motivation of this paper has been to transfer these results in the setting of the Orlicz growth, thus including the full case $1<p<\infty$. In fact, in the quoted papers, the used techniques distinguish the case $1<p\le 2$ and $p>2$.
 
\noindent Problems with general growth conditions, more precisely Orlicz growth, were first studied in \cite{Lieberman, M1, M2, M3} and later in \cite{BV, DE, DLSV, DSV, DSV1, ILV} (see also the references therein for other regularity results in this direction). 

\noindent More recently, starting from the study of autonomous quasi-convex functionals with Orlicz growth in \cite{DLSV},  the case of non-autonomous quasi-convex functionals has been considered in \cite{CO, GSS} (for H\"older and VMO-coefficients, respectively).
In \cite{ok}  the case of homogeneous non-autonomous  systems has been studied, but considering only the case of super-quadratic Orlicz functions in a non-degenerate situation. 
Actually our approach works for arbitrary Orlicz functions, for both the non-degenerate and the degenerate case and it especially works for the full range $1<p<\infty$.


\noindent Moreover we  allow the presence of a forcing term $b$, for which we consider a controllable growth condition (see $(H_7)$ in the subsection 2.2) (see also \cite{Str}).

\noindent 
Normally, the proof of the regularity results for elliptic systems were based on a blow-up technique, but another efficient and by now well established approach for proving partial regularity is based on the so called $\mathcal A$-harmonic approximation method. 
This technique, that originates from De Giorgi \cite{degiorgi} (see also the work of Simon \cite{Simon}),
 has been successfully applied  to obtain partial regularity results for general elliptic systems in a series of papers \cite{DLSV, DG, DGK, DK, DM, DM1}. 
  
 \noindent More precisely, let us consider  a bilinear form $\mathcal A$ on $\mathbb{R}^{nN}$, which is strongly elliptic in the sense of Legendre-Hadamard, with ellipticity constant $\nu > 0$ and upper bound $\Lambda$; that is:   
$$
\nu|\xi|^2|\tilde\xi|^2\le{\mathcal A}(\xi\otimes{ \tilde\xi},\xi\otimes{ \tilde\xi})\leq \Lambda|\xi|^2|{\tilde\xi}|^2 \ \ \ \ \forall\xi\in \mathbb{R}^{n}, {\tilde\xi}\in \mathbb{R}^{N}.
$$
Let us recall that $h$ is called $\mathcal A$-harmonic on a ball $\B_r$ if
$$
\int_{\B_r}{\mathcal A}(Dh,D\eta)dx=0, \ \ \ \forall \eta\in \C_0^\infty(\B_r,\R^N).
$$
The method of $\mathcal A$-harmonic approximation, in the quadratic case, consists in obtaining a good approximation of functions $u \in W^{1,2}(\B_r,\R^N)$ on the ball $\B_r$, which are almost $\mathcal A$-harmonic (in the sense of Theorem \ref{approxdlsv}), by ${\mathcal A}$-harmonic functions $h \in W^{1,2}(\B_r,\R^N)$, in both the $L^2$-topology and in the weak topology of $W^{1,2}$. Thanks to an extension of this technique to the Orlicz setting  it is possible to arrive  to the partial H\"older continuity of the gradients of minimizers to autonomous variational integrals (see \cite{DLSV}). This is obtained proving a
decay estimate for an excess functional which measures in an appropriate sense the oscillations of the gradient of the solution to the problem.

\vskip 0,1cm

\noindent In order to prove our result, we will follow this  $\mathcal A$-harmonic approach in the Orlicz setting.

\noindent
 First of all we will prove a Caccioppoli-type inequality for $u-\ell$ where $u$ is our weak solution and $\ell$ is  a general 
 affine function. In \cite{DLSV}, combining Caccioppoli and Sobolev-Poincar\'e inequalities, it is possible to obtain a higher integrability result that is  fundamental if we want to apply the  $\mathcal A$-harmonic approximation lemma (in the form of  Theorem 14 in \cite{DLSV}). Here the problem is much more complicated due to our Caccioppoli inequality, which takes into account an additional term deriving from our structural assumptions on the function $a$ in the $(x,u)$-variables.  To overcome this difficulty we fully investigated the $\mathcal A$-harmonic approximation lemma of \cite{DLSV}, discovering another useful implication, in terms of the closeness of the functions (see Corollary \ref{coro} below). 
This allowed us to compare  $u-\ell_r$, where $\ell_r$ is the unique affine quasi minimizer of $\dis{\ell\to\-int_{\B_r}|u-\ell |\,dx}$,  with its $\mathcal A$-harmonic approximation $h$, arriving to a suitable decay estimate for an excess functional, that measures the oscillations of $u$ and it is therefore quite appropriate to produce a partial H\"older continuity of $u$. In this manner we avoided  to prove the higher integrability of $Du-D\ell_r$ (with a suitable estimate) which would have necessitated considerably more care, and thus we simplified our proof .

\noindent In this sense our proof differs from the one contained in \cite{GSS} (dealing with minimizers of non-autonomous quasiconvex integrals with VMO-coefficients in the Orlicz setting). We both rely upon comparisons to solutions of linearized systems to establish the decay of some excess functional, in \cite{GSS} it is an ``hybrid'' excess functional that describes both the oscillations of the gradient $Du$ and the oscillations of the function map $u$ and this requires to deal with more technical problems arising when considering the particular form of this excess, as the above mentioned higher integrability of $Du-D\ell_r$.



\section{Assumptions and main result} \label{sect2} 
\noindent
In this section we first provide the notion of $N$-functions, then we give the set of hypotheses and finally we state our main regularity result. 

\subsection{$N$-functions}

We begin recalling the notion of $N$-functions (see \cite{raoren}).

\noindent
We shall say that two real functions $\varphi_1$ and $\varphi_2$ are {\it equivalent}, and we will write $\varphi_1\sim\varphi_2$, if there exist positive constants $c_{1}$ and $c_{2}$ such that $c_{1}\varphi_1(t)\leq \varphi_2(t)\leq c_{2}\varphi_1(t)$ for any $t\ge 0$. 
\begin{defn}\label{defV}
A real function $\varphi: [0, \infty)\rightarrow [0, \infty)$ is said to be an $N$-function if  $\varphi(0)=0$ and 
there exists a right continuous nondecreasing derivative $\V'$ satisfying $\varphi'(0)=0$,
$\varphi'(t)>0$ for $t>0$ and $\displaystyle{\lim_{t\rightarrow \infty} \varphi'(t)=\infty}$. Especially $\varphi$ is convex. 
\end{defn}

\noindent
We say that $\varphi$ satisfies the $\Delta_{2}$-condition (we shall write $\V\in \Delta_{2}$) if there exists a constant $c>0$ such that  
$$
\varphi(2t)\leq c\,\varphi(t) \quad \mbox{ for all } t\geq 0.
$$
We denote the smallest possible constant by $\Delta_{2}(\varphi)$.

\noindent
Since $\varphi(t)\le\varphi(2t)$, the $\Delta_{2}$-condition implies that $\varphi(2t)\sim \varphi(t)$. From now on, we will always denote by $\V$ an $N$-function that satisfies the $\Delta_{2}$-condition. 

\noindent
The conjugate function $\V^{*}:[0, +\infty) \ri [0, +\infty)$ of $\V$ is the function defined by 
\begin{align*}
\V^{*}(t) := \sup_{s\geq 0}\, [st -\V(s)], \quad t\geq 0. 
\end{align*}
Then, also $\V^{*}$ is a $N$-function. If $\V, \V^{*}$ satisfy the $\Delta_{2}$-condition we will write that $\Delta_{2}(\V, \V^{*})<\infty$. 
Assume that $\Delta_{2}(\V, \V^{*})<\infty$. Then for all $\delta>0$ there exists $c_{\delta}$ depending only on $\Delta_{2}(\varphi, \varphi^{*})$ such that
for all $s,t\geq 0$ it holds the Young's inequality
\begin{align*}
&t\, s\leq \delta \, \varphi(t)+c_{\delta} \, \varphi^{*}(s). 
\end{align*}

\noindent Throughout the paper we will assume that $\V$ satisfies the following 
.

\begin{assumption}\label{assumption} 
 Let $\V\in \C^{1}([0, \infty))\cap \C^{2}(0, \infty)$ be an $N$-function satisfying
\begin{align}\label{assnew}
0<p_{0}-1 \leq \inf_{t>0} \frac{t\V''(t)}{\V'(t)}\leq \sup_{t>0} \frac{t\V''(t)}{\V'(t)}\leq p_{1}-1, 
\end{align} 
with $1<p_{0}\leq p_{1}$. Without loss of generality we can assume that $1<p_{0}<2<p_{1}$. 
\end{assumption}

\noindent
Under this assumption on $\V$ it follows from Proposition 2.1 in \cite{CO} that $\Delta_2(\V,\V*)<\infty$ and 
\begin{align}\label{CO1}
p_{0}\leq \inf_{t>0} \frac{t\V'(t)}{\V(t)}\leq \sup_{t>0} \frac{t\V'(t)}{\V(t)}\leq p_{1}. 
\end{align}

\noindent
We point out that the following inequalities holds for every $t\geq 0$:  
\begin{align}\begin{split}\label{t1}
&s^{p_{1}} \V(t) \leq \V(st) \leq s^{p_{0}} \V(t) \quad \mbox{ if } 0< s \leq 1, \\
&s^{p_{0}} \V(t) \leq \V(st) \leq s^{p_{1}} \V(t) \quad \mbox{ if } s \geq 1,
\end{split}\end{align}
as well
\begin{align}\begin{split}\label{t1conj}
&s^{\frac{p_0}{p_0-1}} \V^*(t) \leq \V^*(st) \leq s^{\frac{p_{1}}{p_1-1}} \V^*(t) \quad \mbox{ if } 0< s \leq 1, \\
&s^{\frac{p_{1}}{p_1-1}} \V^*(t) \leq \V^*(st) \leq s^{\frac{p_{0}}{p_0-1}} \V^*(t) \quad \mbox{ if } s \geq 1,
\end{split}\end{align}
and also 
\begin{align}\begin{split}\label{t2}
&s^{p_{1}-1} \V'(t) \leq \V'(st) \leq s^{p_{0}-1} \V'(t) \quad \mbox{ if } 0< s \leq 1, \\
&s^{p_{0}-1} \V'(t) \leq \V'(st) \leq s^{p_{1}-1} \V'(t) \quad \mbox{ if } s \geq 1. 
\end{split}\end{align}
In particular, for $t>0$ we have 
\begin{align}\label{t3}
\V(t)\sim t\V'(t), \quad \V'(t)\sim t\V''(t), \quad \V^{*}(\V'(t))\sim \V^{*}\left(\frac{\V(t)}{t}\right) \sim \V(t). 
\end{align}
For the inverse function $\V^{-1}$ and for $\V' \circ \V^{-1}$ we have the following estimates:
\begin{align}\begin{split}\label{t4}
&s^{\frac{1}{p_{0}}} \V^{-1}(t)\leq \V^{-1}(st) \leq s^{\frac{1}{p_{1}}} \V^{-1}(t), \quad \mbox{ for } 0<s\leq 1\\
&s^{\frac{1}{p_{1}}} \V^{-1}(t)\leq \V^{-1}(st) \leq s^{\frac{1}{p_{0}}} \V^{-1}(t), \quad \mbox{ for } s\geq 1
\end{split}\end{align}
and 
\begin{align}\begin{split}\label{t5conj}
&s^{1-\frac{1}{p_{1}}} (\V^*)^{-1}(t)\leq (\V^*)^{-1}(st) \leq s^{1-\frac{1}{p_{0}}} (\V^*)^{-1}(t), \quad \mbox{ for } 0<s\leq 1\\
&s^{1-\frac{1}{p_{0}}} (\V^*)^{-1}(t)\leq (\V^*)^{-1}(st) \leq s^{1-\frac{1}{p_{1}}} (\V^*)^{-1}(t), \quad \mbox{ for } s\geq 1
\end{split}\end{align}

\noindent
Now, we consider a family of $N$-functions $\{\varphi_{a}\}_{a\geq 0}$ setting, for  $t\geq 0$,
\begin{equation*}\label{phia}
\varphi_{a}(t):=\int_{0}^{t} \varphi'_{a}(s) \, ds \quad \mbox{ with } \quad 
\varphi'_{a}(t):= \varphi'(a+t) \frac{t}{a+t}. 
\end{equation*}
\noindent
The following lemma can be found in \cite{DE} (see Lemma 23 and Lemma 26).

\begin{lem}\label{phia}
Let $\V$ be an $N$-function with $\V\in\Delta_2$ together with its conjugate. Then for all  $a\geq 0$ the function $\varphi_a$ is an $N$-function and $\{\varphi_{a}\}_{a\geq 0}$ and $\{(\varphi_{a})^{*}\}_{a\geq 0}\sim\{\varphi^*_{\varphi'(a)}\}_{a\geq 0}$ satisfy the $\Delta_{2}$ condition
uniformly in $a\geq 0$.
\end{lem}

\noindent
Let us observe that by the previous lemma $\varphi_{a}(t)\sim t\varphi'_{a}(t)$. Moreover, for $t\geq a$ we have $\varphi_{a}(t)\sim \varphi(t)$
and for $t\leq a$ we have $\varphi_{a}(t)\sim t^{2} \varphi''(a)$. This implies that $\varphi_{a}(st)\leq c s^{2} \varphi_{a}(t)$
for all $s\in [0,1]$, $a\geq 0$ and $t\in [0,a]$. 
In particular, the following relations hold uniformly with respect to $a\geq 0$
\begin{align}\begin{split}\label{t6}
&\V_{a}(t)\sim \V''(a+t) t^{2} \sim \frac{\V(a+t)}{(a+t)^{2}}t^{2} \sim \frac{\V'(a+t)}{a+t}t^{2}, \\
&\V(a+t)\sim \V_{a}(t) + \V(a). 
\end{split}\end{align}

\begin{remark}\label{phiap0p1}
It is easy to check that if $\V$ satisfies Assumption \ref{assumption}, the same is true for $\V_a$, uniformly with respect to $a\geq 0$ (with the same $p_0$ and $p_1$).
\end{remark}

\noindent
Next result is a slight generalization of Lemma $20$ in \cite{DE}. 
\begin{lem}\label{lem2}
Let $\varphi$ be an $N$-function with $\Delta_{2}(\{\varphi,\varphi*\})<\infty$; then, uniformly in $z_{1}, z_{2} \in \R^{nN}$  with $|z_{1}|+|z_{2}|>0$, and in $\mu\geq 0$, it holds
\begin{equation*}
 \frac{\varphi'(\mu+|z_{1}|+|z_{2}|)}{\mu +|z_{1}|+|z_{2}|} \sim \int_{0}^{1} \frac{\varphi'(\mu+|z_{\theta}|)}{\mu+|z_{\theta}|} d\theta, 
\end{equation*} 
where $z_{\theta}= z_{1} + \theta (z_{2}-z_{1})$ with $\theta \in [0,1]$.
\end{lem}

\noindent
We shall often use the function $V:\R^{nN}\rightarrow \R^{nN}$ defined by 
\begin{equation}\label{AVlambda1}
V(z) = \sqrt{\frac{\V'(|z|)}{|z|}}z.
\end{equation}
The monotonicity property of $\V$ ensures that 
\begin{align}\label{Vphi}
|V(z_{1})-V(z_{2})|^{2}\sim \V_{|z_{1}|}(|z_{1}-z_{2}|) \quad \mbox{ for any } z_{1}, z_{2} \in \R^{nN},  
\end{align}
and also $|V(z_{1})|^{2}\sim \V(|z_{1}|)$ uniformly in $z_{1}\in \R^{nN}$; see \cite{DE} for further properties about the $V$-function.

\noindent
Let $\V$ be an $N$-function that satisfies the $\Delta_{2}$-condition. The set of functions $L^{\V}(\Omega, \R^{N})$ is defined by
\begin{align*}
L^{\V}(\Omega, \R^{N})= \left\{ u: \Omega \ri \R^{N} \mbox{ measurable } : \, \int_{\Omega} \V(|u|)\, dx <\infty\right\}. 
\end{align*}
The Luxembourg norm is defined as follows:
\begin{align*}
\|u\|_{L^\varphi(\Omega, \R^{N})}=\inf \left\{\lambda>0 : \int_{\Omega} \varphi \left(\frac{|u(x)|}{\lambda} \right)\,dx\leq 1\right\}.
\end{align*}
With this norm $L^\varphi(\Omega, \R^{N})$ is a Banach space. 

\noindent
By $W^{1, \V}(\Omega, \R^{N})$ we denote the classical Orlicz-Sobolev space, that is $u\in W^{1, \V}(\Omega, \R^{N})$ whenever $u,  Du \in L^{\V}(\Omega, \R^{N})$.
Furthermore, by $W^{1,\V}_{0}(\Omega, \R^{N})$ we mean the closure of $\C^{\infty}_{c}(\Omega, \R^{N})$ functions with respect to the norm
\begin{align*}
\|u\|_{W^{1, \V}(\Omega, \R^{N})}=\|u\|_{L^{\V}(\Omega, \R^{N})}+\| Du\|_{L^{\V}(\Omega, \R^{N})}. 
\end{align*}

\noindent
The following version of the Sobolev-Poincar\'e inequality can be found in \cite{BV}. 
\begin{thm}\label{SP}
Let $\V$ be an $N$-function with $\Delta_{2}(\V, \V^{*})<\infty$ and let $\B_{r}\subset \R^{n}$. Then for every $1\le s< \frac{n}{n-1}$, there exists $c_p>0$ such that, for all $u\in W^{1, \V}(\B_{r}, \R^{N})$, it holds 
\begin{align*}
\left(\-int_{\B_{r}} \V^{s}\left( \frac{|u-(u)_{r}|}{r}\right)\, dx\right)^{\frac{1}{s}} \leq c_p \-int_{\B_{r}} \V(|Du|)\, dx.
\end{align*}
\end{thm}

\subsection{Assumptions on $a$ and $b$}\label{hpoperatore}
For the vector field $a:\Omega \times \R^{N}\times \R^{nN}\ri \R^{nN}$ we assume that it is Borel measurable and that the partial map $\xi\to a(\cdot,\cdot,\xi)$ is differentiable. Moreover we consider the following set of assumptions.
\begin{compactenum}[$(H_{1})$]
\item [$(H_{1})$] $|a(x, u, \xi)|\leq \Lambda \V'(|\xi|)$, for any $x\in \Omega$, $u\in \R^{N}$ and $\xi \in \R^{nN}$; 
\item [$(H_{2})$] $(D_{\xi}a(x, u, \xi)\eta, \eta)\geq \nu \V''(|\xi|)|\eta|^{2}$, for any $x\in \Omega$, $u\in \R^{N}$, $\xi, \eta\in \R^{nN}$, $\xi\neq 0$, $0<\nu\leq 1\leq \Lambda<\infty$; 
\item [$(H_{3})$] $|D_{\xi} a(x, u, \xi)|\leq \Lambda \V''(|\xi|)$, for any $x\in \Omega$, $u\in \R^{N}$ and $\xi \in \R^{nN}$, $\xi\neq 0$;  
\item [$(H_{4})$] $\forall \xi, \eta\in \R^{nN}$ such that $0<|\eta|\leq \frac{1}{2}|\xi|$ it holds
\begin{align*}
|D_{\xi}a(x, u, \xi) - D_{\xi}a(x, u, \xi+\eta)|\leq \Lambda \V''(|\xi|) \left(\frac{|\eta|}{|\xi|}\right)^{\beta_{0}}
\end{align*}
for some $\beta_{0} \in (0, 1]$;
\item [$(H_{5})$] $|a(x, u, \xi) -a(x, u_{0}, \xi)|\leq \Lambda \omega(|u-u_{0}|) \V'(|\xi|)$ where $\omega: [0, \infty)\ri [0, 1]$ is a non-decreasing concave modulus of continuity with $\displaystyle{\lim_{s\ri 0} \omega(s)=0=\omega(0)}$; 
\item [$(H_{6})$] the following VMO-condition holds true: 
\begin{align*}
|a(x, u, \xi) - \left( a(\cdot, u, \xi)\right)_{x_{0},r}|\leq \Lambda \w_{x_{0}}(x, r) \V'(|\xi|) \mbox{ for all } x\in \B_{r}(x_{0})
\end{align*} 
whenever $x_{0}\in \Omega$, $r\in (0, \rho_{0}]$, $u\in \R^{N}$ and $\xi\in \R^{nN}$, where $\rho_{0}>0$ and $\w_{x_{0}}: \R^{n}\times [0, \rho_{0}]\ri [0, 2\Lambda]$ is a  bounded functions satisfying 
\begin{align*}
\lim_{\rho \ri 0} {\rm W}(\rho) =0, \mbox{ where } {\rm W}(\rho)= \sup_{x_{0}\in \Omega} \, \sup_{0<r\leq \rho} \, \-int_{\B_{r}(x_{0})} \w_{x_{0}}(x, r)\, dx. 
\end{align*}
\end{compactenum}
We will use the notation
\begin{align*}
\left( a(\cdot, u, \xi)\right)_{x_{0},r}:=\-int_{\B_{r}(x_{0})} a(y, u, \xi)\, dy. 
\end{align*}
In the special case $x_{0}=0$ we omit the dependence on $x_{0}$ and we will simply write $\left( a(\cdot, u, \xi)\right)_{r}$. 
\smallskip

\noindent Moreover we assume that the vector field $a(x,u,\cdot)$ admits a $\V$-Laplacian type behavior at the origin in the sense that the limit relation
\begin{equation}\label{nearzero}
\lim_{t\to 0^+}\frac{a(x,u,t\xi)}{\V'(t)}=|\xi|
\end{equation}
holds uniformly in $\{\xi\in\R^{Nn} : |\xi|=1\}$, and uniformly for all  $x\in\Omega$ and all $u\in\R^N$.
\smallskip

\noindent The inhomogeneity $b:\Omega\times \R^{N}\times \R^{nN}\ri \R^{N}$ is a Borel measurable function verifying the following controllable growth condition:
\begin{compactenum}
\item [$(H_{7})$] $|b(x, u, \xi)|\leq L \V'(|\xi|)$, for any $x\in \Omega$, $u\in \R^{N}$ and $\xi \in \R^{nN}$.
\end{compactenum}

\medskip

\noindent
Here, we consider weak solutions $u\in W^{1, \V}(\Omega, \R^{N})$ of elliptic systems of the type \eqref{P}. By this we mean
\begin{align}\label{ws}
\int_{\Omega} a(x, u, Du) D\eta \, dx = \int_{\Omega} b(x, u, Du) \eta \, dx
\end{align}
for any $\eta \in \C^{\infty}_{0}(\Omega, \R^{N})$. 

\smallskip

\noindent We define $\Sigma_1$ and $\Sigma_2$ the sets 
$$
\Sigma_{1}=\left\{x_{0} \in \Omega \, : \, \liminf_{r\ri 0} \-int_{\B_{r}(x_{0})} |V(Du) -(V(Du))_{x_{0}, r}|^{2}\, dx >0\right\}, 
$$
$$
\Sigma_2=\left\{x_0\in\Omega : \limsup_{r\to 0} |(Du)_{x_{0},r}|=+\infty \right\},
$$
where $V$ is in (\ref{AVlambda1}). The regular points  will be the set of full Lebesgue measure $\mathcal{R}(u)=\Omega\setminus (\Sigma_1\cup\Sigma_2)$. It turns out that $\mathcal{R}(u)\subset\Omega_0$, where $\Omega_0$ is given  in the following theorem which is our main result.

\begin{thm}\label{mainthm}
Let $\V$ satisfy Assumption \ref{assumption} and let $u\in W^{1, \V}(\Omega, \R^{N})$ be a weak solution of \eqref{P} under the hypotheses $(H_{1})$-$(H_{7})$ and \eqref{nearzero}. Then there exists an open subset $\Omega_0\subseteq\Omega$ such that $u\in \C^{0,\alpha}(\Omega_0,\R^N)$ for every $\alpha\in (0,1)$ and $|\Omega\setminus\Omega_0|=0$.
 \end{thm}

\subsection{Notations}
In order to simplify the presentation, we denote by $c$ a generic positive constant, which may change from line to line, but does not depend on crucial quantities. By $\B_{r}(x_{0})$ we indicate the open ball in $\R^{n}$ centered in $x_{0}\in \R^{n}$ and radius $r>0$. In the case $x_{0}=0$ we simply write $\B_{r}$. For $v\in L^{1}(\B_{r}(x_{0}))$, we define
\begin{align*}
(v)_{x_{0}, r}:=\-int_{\B_{r}(x_{0})} v(x) \, dx,  
\end{align*} 
and when $x_{0}=0$ we omit the dependance on $x_{0}$ as follows $(v)_{r}=(v)_{0, r}$.


\section{Preliminary}

\noindent
We collect here some useful result that will be needed in the sequel.

\subsection{Affine functions}

\noindent
Let $x_{0}\in \R^{n}$ and $r>0$. Given $u\in L^{2}(\B_{r}(x_{0}), \R^{N})$, we denote by $\ell_{x_{0}, r}:\R^{n}\ri \R^{N}$ the unique affine function minimizing the functional
\begin{align*}
\ell\mapsto \-int_{\B_{r}(x_{0})} |u-\ell|^2\, dx 
\end{align*}
amongst all affine functions $\ell: \R^{n}\ri \R^{N}$. It is well known that 
\begin{align*}
\ell_{x_{0}, r}(x)= (u)_{x_{0}, r} + Q_{x_{0}, r}(x-x_{0}), 
\end{align*}
where 
\begin{align*}
Q_{x_{0}, r}= \frac{n+2}{r^{2}} \-int_{\B_{r}(x_{0})} u(x)\otimes (x-x_{0})\, dx.  
\end{align*}
Let us recall that for any $Q\in\R^{Nn}$ and $\xi\in\R^N$ there holds
\begin{equation}\label{minimizza}
|Q_{x_{0}, r}-Q|\le\frac{n+2}{r} \-int_{\B_{r}(x_{0})} |u-\xi-Q(x-x_0)|\, dx.
\end{equation}
\noindent The following lemma ensures that $\ell_{x_0,r}$ is an almost minimizer of the functional $\displaystyle{\ell\mapsto \-int_{\B_{r}(x_{0})} |u-\ell|\, dx}$ amongst the affine functions $\ell: \R^{n}\ri \R^{N}$ (see Lemma 2.7 in \cite{BJDE}).

\begin{lem}
Let $u\in L^{1}(\B_{r}(x_{0}), \R^{N})$. Then, we have
$$
 \-int_{\B_{r}(x_{0})} |u-\ell_{x_0,r}|\, dx\le c  \-int_{\B_{r}(x_{0})} |u-\ell|\, dx,
 $$
 for every  affine function $\ell: \R^{n}\ri \R^{N}$.
\end{lem}

\noindent Actually we have that  $\ell_{x_0,r}$ is an almost minimizer also of the functional $\displaystyle{\ell\mapsto \-int_{\B_{r}(x_{0})} \V_a\left(\frac{|u-\ell|}{r}\right)\, dx}$ amongst the affine functions $\ell: \R^{n}\ri \R^{N}$. 
\begin{lem}\label{minaffine}
Let $a\geq 0$ and $\B_{r}(x_{0})\subset \R^{n}$ with $r>0$. Then, for any $u\in W^{1, \V}(\B_{r}(x_{0}), \R^{N})$ we have
\begin{align*}
\-int_{\B_{r}(x_{0})} \V_{a}\left(\frac{|u-\ell_{x_{0},r}|}{r}\right)\, dx \leq c \-int_{\B_{r}(x_{0})} \V_{a}\left(\frac{|u-\ell|}{r}\right)\, dx 
\end{align*}
where $c= c(p_{0}, p_{1})>0$. 
\end{lem}

\begin{proof}
Assume $x_{0}=0$. Recalling that $\V_a(s+t)\sim \V_a(s)+ \V_a(t)$ for any $s, t\geq 0$, we have 
\begin{align*}
\-int_{\B_{r}} \V_{a}\left(\frac{|u-\ell_{r}|}{r}\right)\, dx \le c \-int_{\B_{r}} \V_{a}\left(\frac{|u-\ell|}{r}\right)\, dx + c\-int_{\B_{r}} \V_{a}\left(\frac{|\ell-\ell_{r}|}{r}\right)\, dx. 
\end{align*}
We note that for any $x\in \B_{r}$, we get
\begin{align*}
|\ell_{r}(x)-\ell(x)|\leq |(u)_{r}-\ell(0)|+ |D\ell_{r}- D\ell|r \leq c \-int_{\B_{r}} |u-\ell|\, dx.  
\end{align*}
Hence, using this inequality, the fact that $\V_{a}$ is increasing together with Jensen's inequality, we can infer that
\begin{align*}
\-int_{\B_{r}} \V_{a} \left( \frac{|\ell- \ell_{r}|}{r}\right) \, dx &\le c \-int_{\B_{r}} \V_{a} \left(\frac{1}{r} \-int_{\B_{r}} |u-\ell|\, dx\right)\, dx \\
&\lesssim \-int_{\B_{r}} \V_{a} \left( \frac{|u-\ell|}{r}\right)\, dx. 
\end{align*}
\end{proof}

\noindent The following lemma is proved in \cite[Corollary 26]{DKr}.
\begin{lem}\label{shift}
Let $\V$ be a $N$-function satisfying Assumption \ref{assumption}. Then for any $\delta>0$ there exists $C_\delta > 0$, which only depends on $\delta$ and $\Delta_2(\V)$ such that for all $a,b\in\R^d$ and $t\ge 0$
$$
\V_{|a|}(t)\le C_\delta \V_{|b|}(t) + \delta \V_{|a|}(|a-b|).
$$
\end{lem}

\begin{remark}\label{mediamin}
Another basic inequality is the following:
\begin{align}\label{ok2.7}
\-int_{\B_{r}(x_{0})} \V_{a} \left( \frac{|u-(u)_{x_{0}, r}|}{R}\right) \leq c \,\-int_{\B_{r}(x_{0})} \V_{a}\left(\frac{|u-u_{0}|}{R}\right)\, dx.
\end{align}
for any $u_0\in\R^N$ and for any $a,R>0$.
\end{remark}
\subsection{Useful remarks} 

\begin{remark}\label{rem1}
For any $x\in \Omega$, $u\in \R^{N}$ and $\xi, \xi_{0}\in \R^{nN}\setminus \{0\}$, using $(H_2)$, \eqref{t3} and \eqref{t6} we obtain
\begin{align*}
\left(a(x, u, \xi)- a(x, u, \xi_{0}), \xi-\xi_{0}\right)&= \int_{0}^{1}\frac{d}{dt} \left( a(x, u, \xi_{0}+t(\xi-\xi_{0})), \xi-\xi_{0}\right)\, dt \\
&=\int_{0}^{1} \left(D_{\xi} a(x, u, \xi_{0}+t(\xi-\xi_{0})) (\xi-\xi_{0}), (\xi-\xi_{0})\right)\, dt \\
&\geq \nu \int_{0}^{1} \V''(|\xi_{0}+t(\xi-\xi_{0})|) \, |\xi-\xi_{0}|^{2}\, dt \\
&\ge c \int_{0}^{1} \frac{\V'(|\xi_{0}+t(\xi-\xi_{0})|)}{|\xi_{0}+t(\xi-\xi_{0})|} \, |\xi-\xi_{0}|^{2}\, dt \\
&\ge c\frac{\V'(|\xi_{0}|+|\xi|)}{|\xi_{0}|+|\xi|} \, |\xi-\xi_{0}|^{2}\\
&\geq c \frac{\V'(|\xi_{0}|+|\xi_{0}-\xi|)}{|\xi_{0}|+|\xi_{0}-\xi|} |\xi-\xi_{0}|^{2}\\
&\geq c \V_{|\xi_{0}|}(|\xi-\xi_{0}|).
\end{align*}
\end{remark}

\begin{remark}\label{rem2}
Following the lines of Remark \ref{rem1} we can prove that for any $x\in \Omega$, $u\in \R^{N}$ and $\xi, \xi_{0}\in \R^{nN}$ it holds
\begin{align*}
|a(x, u, \xi)- a(x, u, \xi_{0})|\leq c \V'_{|\xi_{0}|}(|\xi-\xi_{0}|). 
\end{align*}
\end{remark}

\subsection{The $\varphi$-harmonic approximation}

We will use the following $\V$-harmonic approximation result (see \cite[Theorem 1.1]{DSV1}).

\begin{thm}\label{phiappr}
Let $\V$ satisfy Assumption \ref{assumption}; for every $\e>0$ and $\theta\in (0,1)$ there exists $\delta_0=\delta_0(n,N,p_0,p_1,\e,\theta)$ such that the following holds:
Whenever $u\in W^{1\V}(\B_{2r}(x_0),\R^N)$ is almost $\V$-harmonic in the ball $\B_r(x_0)$ in the sense that
$$
\-int_{\B_r(x_0)} \left(\V'(|Du|)\frac{Du}{|Du|},D\zeta \right)\,dx\le\delta_0\left(\-int_{\B_{2r}(x_0)}\V(|Du|)\,dx+\V(\|D\zeta\|_{\infty})\right),
$$
for all $\zeta\in \C^\infty_c(\B_r(x_0),\R^N)$, then the unique $\V$-harmonic solution $h$ of 
\begin{align*}
\left\{
\begin{array}{ll}
-\dive \left( \V'(|Dh|) \frac{Dh}{|Dh|}\right) =0 &\mbox{ in } \B_r(x_0), \\
h=u &\mbox{ on } \partial \B_r(x_0) 
\end{array}
\right. 
\end{align*}
satisfies
\begin{align*}
\left(\-int_{\B_r(x_0)} |V(Du)- V(Dh)|^{2\theta} \, dx \right)^{\frac{1}{\theta}} \leq \e \-int_{\B_{2r}(x_0)} \V(|Du|)\, dx.
\end{align*}
\end{thm}

\noindent The following results are contained in \cite{DSV} (see Lemma 5.8 and Theorem 6.4). 

\begin{thm}\label{regh}
Let $\Omega\subset\R^n$ be an open set, let $\V$ satisfy Assumption \ref{assumption}, and let $h\in W^{1,\V}(\Omega,\R^N)$ be $\V$-Harmonic on $\Omega$. Then for every ball $\B$ with 
$2\B\Subset\Omega$ there holds
\begin{equation}\label{suph}
\sup_{\B} \V(|Dh|)\le c_*\-int_{2\B}\V(Dh|)\,dx,
\end{equation}
where $c_*$ depends on $n,N,p_0,p_1$. 

\noindent Moreover, there exist  $\alpha_0>0$ and $c^*=c^*(n,N,p_0,p_1)$ such that for every ball $\B\subset\Omega$ and every $\lambda\in (0,1)$ there holds
\begin{equation}\label{decayh}
\-int_{\lambda\B}|V(Dh)-(V(Dh))_{\lambda\B}|^2 dx\le c^*\lambda^{2\alpha_0}\-int_{\B}|V(Dh)-(V(Dh))_{\B}|^2 dx.
\end{equation}
\end{thm}

\subsection{$\mathcal{A}$-harmonic approximation}

In our context, we shall exploit a suitable version of the $\mathcal{A}$-harmonic approximation result given in \cite{DLSV}. It can be retrieved through a slight modification of Theorem 14 in \cite{DLSV}. To do this we will need the following lemma.

\begin{lem}\label{phiausilio}
Let $\V$ satisfy Assumption \ref{assumption} and let $1<s<\frac{p_{1}}{p_{1}-p_{0}+1}$ and $\psi(t)= \V^{\frac{1}{s}}(t)$ for $t\geq 0$. Then $\psi$ satisfies Assumption \ref{assumption} with $p_{0}$ and $p_{1}$ replaced by $\left(\frac{1-s}{s}\right)p_{1}+ p_{0}$ and $\left(\frac{1-s}{s}\right) p_{0}+ p_{1}$, respectively. 
\end{lem}

\begin{proof}
The function $\psi$ is  $\C^{2}(0, \infty)$ and $\psi(0)=0$. On the other hand, its derivative is: 
\begin{align*}
\psi'(t)= \frac{1}{s} \V^{\frac{1}{s}-1}(t) \V'(t)>0 \quad \forall t>0.  
\end{align*} 
We can easily deduce that $\displaystyle{\lim_{t\ri 0^{+}} \psi'(t)=0}$. Indeed, using \eqref{CO1} and observing that $p_{1}>s$, for any $t\in (0,1)$, we have 
\begin{align*}
\psi'(t)= \frac{1}{s} \frac{\V^{\frac{1}{s}}(t)}{t} \frac{t\V'(t)}{\V(t)} \leq \frac{p_{1}}{s} \frac{t^{\frac{p_{1}}{s}} \V(1)^{\frac{1}{s}}}{t} \ri 0  \mbox{ as } t\ri 0^{+}.
\end{align*}
In a similar way we can prove that $\displaystyle{\lim_{t\ri +\infty} \psi'(t)=+\infty}$, since $s<p_0$. 
By the choice of $s$  we can also show that $\psi$ satisfies 
\begin{align*}
\left(\frac{1-s}{s}\right)p_{1}+ p_{0}-1 \leq \inf_{t>0} \frac{t\, \psi''(t)}{\psi'(t)} \leq \sup_{t>0} \frac{t\, \psi''(t)}{\psi'(t)} \leq \left(\frac{1-s}{s}\right) p_{0}+ p_{1}-1.
\end{align*}
Indeed, by using \eqref{assnew} and \eqref{CO1}, for $t>0$ we have
\begin{align*}
\frac{t\, \psi''(t)}{\psi'(t)}&= \frac{t \left[\frac{1-s}{s} \V'(t)^{2} + \V(t)\V''(t)\right]}{\V(t)\V'(t)}\\
&= \left(\frac{1-s}{s}\right) \frac{t\V'(t)}{\V(t)} + \frac{t\V''(t)}{\V'(t)}\left\{
\begin{array}{ll}
\leq \left(\frac{1-s}{s}\right) p_{0}+ p_{1}-1\\
\geq \left(\frac{1-s}{s}\right)p_{1}+ p_{0}-1,
\end{array}
\right.
\end{align*}
and, of course, $\left(\frac{1-s}{s}\right)p_{1}+ p_{0}-1>0$.
\end{proof}

\noindent
As said before, our $\mathcal{A}$-harmonic approximation result is a slight modification of Theorem 14 in \cite{DLSV}. We prove it for the sake of completeness.

\begin{thm}\label{approxdlsv}
Let $\B_r(x_0)\Subset \Omega$ a ball let $\A$ be strongly elliptic in the sense of Legendre-Hadamard. Let $\V$ be an $N$-function satisfying Assumption \eqref{assumption} and let $1<s<\frac{p_{1}}{p_{1}-p_{0}+1}$. Then, for every $\e>0$ there exists $\delta_{0}=\delta_0(n, N, \lambda, \Lambda, \Delta_{2}(\V, \V^{*}), s,\e)>0$ such that the following holds: let $u\in W^{1, \V}(\B_r(x_0),\R^N)$ be almost $\A$-harmonic in the sense that 
\begin{align}\label{assdlsv} 
\left| \-int_{\B_r(x_0)} \A(Du, D\xi) \, dx\right|\leq \delta_{0} \Upsilon \|D\xi\|_{\infty} 
\end{align}   
for all $\xi \in \C^{\infty}_{c}(\B_r(x_0),\R^N)$. 
Given $h$ solution to
\begin{align*}
\left\{
\begin{array}{ll}
-\dive(\A Dh)=0 &\mbox{ on } \B_r(x_0), \\
h=u &\mbox{ on } \partial \B_r(x_0).
\end{array}
\right. 
\end{align*}
Then $h$ satisfies 
\begin{align*}
\-int_{\B_r(x_0)} \V^{\frac{1}{s}} (|D(u-h)|)\, dx \leq \e \left\{\left[ \-int_{\B_r(x_0)} \V(|Du|)\, dx \right]^{\frac{1}{s}} + \V^{\frac{1}{s}}(\Upsilon) \right\}. 
\end{align*}
\end{thm}

\begin{proof}
Let $\psi$ be an $N$-function satisfying Assumption \ref{assumption}. Then, 
from Lemma 20 in \cite{DLSV} it follows that for any $v\in W^{1, \psi}_{0}(\B_r(x_0),\R^N)$ it holds
\begin{align}\label{dlsv1}
\-int_{\B_r(x_0)} \psi(|Dv|) \sim \sup_{\xi\in \C^{\infty}_{c}(\B_r(x_0),\R^N)} \left[ \-int_{\B_r(x_0)} \A(Dv, D\xi) \, dx - \-int_{\B_r(x_0)} \psi^{*}(|D\xi|)\, dx \right]. 
\end{align}
We will choose $\psi= \V^{\frac{1}{s}}$, recalling that such $\psi$ satisfies Assumption \ref{assumption}, thanks to Lemma \ref{phiausilio}. Let $\xi \in \C^{\infty}_{c}(\B_r(x_0),\R^N)$ and choose $\mu\geq 0$ such that 
\begin{align}\label{phi*mu}
\psi^{*}(\mu)= \-int_{\B_r(x_0)} \psi^{*}(|D\xi|)
\end{align}
and take $m_{0}\in \N$. Then, from Theorem 21 in \cite{DLSV} there exists $\lambda \in [\mu, 2^{m_{0}}\mu]$ and $\xi_{\lambda}\in W^{1, \infty}_{0}(\B_r(x_0),R^N)$ such that 
\begin{align}
&\|D\xi_{\lambda}\|_{\infty}\leq c \lambda \label{dlsv2} \\
&\-int_{\B_r(x_0)} \psi^{*}(|D\xi_{\lambda}|) \chi_{\{\xi \neq \xi_{\lambda}\}} \, dx \leq c \psi^{*}(\lambda) \frac{|\{\xi\neq \xi_{\lambda}\}|}{|\B_r(x_0)|} \leq \frac{c}{m_{0}} \-int_{\B_r(x_0)} \psi^{*}(|D\xi|)\, dx \label{dlsv3} \\
&\-int_{\B_r(x_0)} \psi^{*}(|D\xi_{\lambda}|) \, dx \leq c\-int_{\B_r(x_0)} \psi^{*}(|D\xi|)\, dx. \label{dlsv4}
\end{align}
Now, we note that 
\begin{align*}
\-int_{\B_r(x_0)} \A(Du, D\xi)\, dx = \-int_{\B_r(x_0)} \A(Du, D\xi_{\lambda})\, dx + \-int_{\B_r(x_0)} \A(Du, D(\xi-\xi_{\lambda}))\, dx= I+ I\!I.  
\end{align*}
Let us study $I\!I$. Using the growth assumption, Young inequality together with \ref{dlsv4} and $\lambda \in [\mu, 2^{m_{0}}\mu]$ we get
\begin{align*}
I\!I&= \-int_{\B_r(x_0)} \A(Du, D(\xi-\xi_{\lambda})) \chi_{\{\xi\neq \xi_{\lambda}\}}\, dx\leq \Lambda \-int_{\B_r(x_0)} |Du| |D(\xi-\xi_{\lambda})| \chi_{\{\xi\neq \xi_{\lambda}\}}\, dx \\
&\leq c\-int_{\B_r(x_0)} \psi(|Du|\chi_{\{\xi\neq \xi_{\lambda}\}})\, dx +\frac{1}{2} \-int_{\B_r(x_0)}\psi^{*}(|D\xi|)\, dx \\
&\leq c \left( \-int_{\B_r(x_0)} \psi(|Du|)^{s} \, dx\right)^{\frac{1}{s}} \left(\frac{|\{\xi\neq \xi_{\lambda}\}|}{|\B_r(x_0)|} \right)^{1-\frac{1}{s}} + \frac{1}{2} \-int_{\B_r(x_0)}\psi^{*}(|D\xi|)\, dx \\
&\leq c \left( \-int_{\B_r(x_0)} \psi(|Du|)^{s} \, dx\right)^{\frac{1}{s}} \left( \frac{\psi^{*}(\mu)}{\psi^{*}(\lambda)}\frac{1}{m_{0}}\right)^{1-\frac{1}{s}} + \frac{1}{2} \-int_{\B_r(x_0)}\psi^{*}(|D\xi|)\, dx \\
&\leq c \left( \-int_{\B_r(x_0)} \psi(|Du|)^{s} \, dx\right)^{\frac{1}{s}} \left(\frac{1}{m_{0}}\right)^{1-\frac{1}{s}}+ \frac{1}{2} \-int_{\B_r(x_0)}\psi^{*}(|D\xi|)\, dx \\
&\leq \e \left( \-int_{\B_r(x_0)} \psi(|Du|)^{s} \, dx\right)^{\frac{1}{s}}  + \frac{1}{2} \-int_{\B_r(x_0)}\psi^{*}(|D\xi|)\, dx
\end{align*}
provided that $m_{0}$ is sufficiently large. 

\noindent
Let us estimate $I$. From \ref{assdlsv}, \ref{dlsv2} and $\lambda \in [\mu, 2^{m_{0}}\mu]$, we infer
\begin{align*}
\left| \-int_{\B_r(x_0)} \A (Du, D\xi_{\lambda})\, dx \right|\leq \delta_{0} \Upsilon \|D\xi_{\lambda}\|_{\infty} \leq \delta_{0} \Upsilon c \lambda \leq c \delta_{0} \Upsilon 2^{m_{0}} \mu \leq c \delta_{0} 2^{m_{0}} \left[\psi(\Upsilon) + \psi^{*}(\mu)\right]. 
\end{align*}
Now if $\delta_{0}$ is such that $2 c \delta_{0} 2^{m_{0}}\leq \e$, from \ref{phi*mu}, we deduce 
\begin{align*}
|I|\leq \e \psi(\Upsilon) +\frac{1}{2} \-int_{\B_r(x_0)} \psi^{*}(|D\xi|)\, dx.  
\end{align*}
Therefore
\begin{align*}
\-int_{\B_r(x_0)} \A(Du, D\xi)\, dx \leq \e \left[ \left( \-int_{\B_r(x_0)} \psi(|Du|)^{s} \, dx\right)^{\frac{1}{s}} + \psi(\Upsilon) \right] +\-int_{\B_r(x_0)} \psi^{*}(|D\xi|)\, dx.  
\end{align*}
Recalling that $h$ is $\A$-harmonic we have
\begin{align*}
\-int_{\B_r(x_0)} \A(D(u-h), D\xi)\, dx - \-int_{\B_r(x_0)} \psi^{*}(|D\xi|)\, dx \leq \e \left[ \left( \-int_{\B_r(x_0)} \psi(|Du|)^{s} \, dx\right)^{\frac{1}{s}} + \psi(\Upsilon) \right].  
\end{align*}
Exploiting (\ref{dlsv1}) and $u-h=0$ on $\partial \B_r(x_0)$ we obtain 
\begin{align*}
\-int_{\B_r(x_0)} \psi(|D(u-h)|)\, dx \leq \e \left[ \left( \-int_{\B_r(x_0)} \psi(|Du|)^{s} \, dx\right)^{\frac{1}{s}} + \psi(\Upsilon) \right]
\end{align*}
from which follows the thesis thanks to the choice of $\psi$. 
\end{proof}

\noindent
Applying Theorem \ref{SP} we get the following result.

\begin{cor}\label{coro}
In the same hypotheses of Theorem \ref{approxdlsv}, if $1<s<\min \left\{ \frac{n}{n-1}, \frac{p_{1}}{p_{1}-p_{0}+1} \right\}$, then 
\begin{align*}
\left(\-int_{\B_r(x_0)} \V\left( \frac{|u-h|}{r}\right)\, dx \right)^{\frac{1}{s}} \leq \e \left\{ \left(\-int_{\B_r(x_0)} \V(|Du|)\, dx \right)^{\frac{1}{s}} + \V^{\frac{1}{s}}(\Upsilon)\right\}.
\end{align*}
\end{cor}

\section{A Caccioppoli-type inequality}

\noindent
A standard preliminary tool used to obtain partial regularity is the Caccioppoli inequality. Here we prove a version of this inequality for weak solutions to \eqref{P} taking into account the dependance of the function $a$ on the coefficients $(x, u)$.

\begin{thm}\label{Caccioppoli}
Let $u\in W^{1, \V}(\Omega, \R^{N})$ be a weak solution to \eqref{P} under the same assumptions of Theorem \ref{mainthm}. Then, for every ball $\B_{r}(x_{0})\Subset \Omega$ and for every affine function $\ell: \R^{n}\ri \R^{N}$ defined by 
\begin{align*}
\ell (x)= Q(x-x_{0}) + y_{0}, \quad x\in \R^{n}, 
\end{align*}
with $Q\in \R^{nN}$ and $y_{0}\in \R^{N}$, the following inequality holds
\begin{align*}
&\-int_{\B_{\frac{r}{2}}(x_{0})} \V_{|D\ell|}(|Du-D\ell|) \, dx \\
&\quad \leq c_{cacc}\left\{ \-int_{\B_{r}(x_{0})} \V_{|D\ell|} \left(\frac{|u-\ell|}{r}\right) \, dx+\V(|D\ell|) \left[ {\rm W}(r)+ \omega \left( \-int_{\B_{r}(x_{0})} |u-\ell(x_{0})|\, dx\right) +r\right]\right\}
\end{align*}
\end{thm}

\begin{proof}
Assume $x_{0}=0$. Take $\zeta$ a cut off function between $\B_{\frac{r}{2}}$ and $\B_{r}$, and define 
\begin{align*}
\eta= \zeta^{p_{1}}(u-\ell). 
\end{align*} 
Using $\eta$ as test function in \eqref{ws} we obtain
\begin{align}\begin{split}\label{eqc1}
&\-int_{\B_{r}} a(x, u, Du)\zeta^{p_{1}}(Du-D\ell)\, dx \\
&\qquad= -p_{1}\-int_{\B_{r}} a(x, u, Du) \zeta^{p_{1}-1}  (u-\ell)\otimes D\zeta \, dx + \-int_{\B_{r}} b(x, u, Du) \zeta^{p_{1}} (u-\ell)\, dx.  
\end{split}\end{align}
First, let us observe that 
\begin{align*}
\-int_{\B_{r}} (a(\cdot, \ell(0), D\ell))_{r} D\eta \, dx =0. 
\end{align*}
So \eqref{eqc1} can be rewritten as 
\begin{align*}
I&:=\-int_{\B_{r}} (a(x, u, Du)-a(x, u, D\ell), (Du-D\ell)) \zeta^{p_{1}}\, dx \\
&=-p_{1} \-int_{\B_{r}} (a(x, u, Du)-a(x, u, D\ell))\zeta^{p_{1}-1} (u-\ell)\otimes D\zeta dx - \-int_{\B_{r}} (a(x, u, D\ell)-a(x, \ell(0), D\ell), D\eta)\, dx \\
&\qquad - \-int_{\B_{r}} (a(x, \ell(0), D\ell)-(a(\cdot, \ell(0), D\ell))_{r}, D\eta)\, dx + \-int_{\B_{r}} b(x, u, Du) \zeta^{p_{1}} (u-\ell)\, dx \\
&=: I\!I+ I\!I\!I+ I\!V+ V. 
\end{align*}
Exploiting Remark \ref{rem1} we can see that 
\begin{align*}
I\geq \nu c \-int_{\B_{r}} \V_{|D\ell|} (|Du-D\ell|)\zeta^{p_{1}}\, dx, 
\end{align*}
where $c$ depends on $p_{0}, p_{1}$. 

\noindent
Combining Remark \ref{rem2} together with Young's inequality, \eqref{t1conj} and \eqref{t3} we obtain
\begin{align*}
I\!I&\leq c \Lambda \-int_{\B_{r}} \V'_{|D\ell|}(|Du-D\ell|) \zeta^{p_{1}-1} |D\zeta| |u-\ell|\, dx \\
&\le\delta \-int_{\B_{r}} (\V_{|D\ell|})^{*} \left(\V'_{|D\ell|}(|Du-D\ell|) \zeta^{p_{1}-1}\right)\, dx + c(\delta) \-int_{\B_{r}} \V_{|D\ell|} \left( \frac{|u-\ell|}{r}\right)\, dx \\
&\leq c\delta \-int_{\B_{r}} \V_{|D\ell|}(|Du-D\ell|) \zeta^{p_{1}} \, dx + c(\delta) \-int_{\B_{r}} \V_{|D\ell|} \left( \frac{|u-\ell|}{r}\right)\, dx.
\end{align*}
Now we estimate $I\!I\!I$, using $(H_{5})$ and the definition of $\eta$, so that
\begin{align*}
I\!I\!I \leq \Lambda \-int_{\B_{r}} \omega(|u-\ell(x_{0})|) \V'(|D\ell|) \left[ p_{1}\zeta^{p_{1}-1} |u-\ell| |D\zeta| + \zeta^{p_{1}} |Du-D\ell|\right]\, dx = I\!I\!I_{1}+ I\!I\!I_{2}.  
\end{align*}
Again the use of Young's inequality gives
\begin{align*}
I\!I\!I_{1}&\leq \delta \-int_{\B_{r}}(\V_{|D\ell|})^{*}(\omega(|u-\ell(0)|) \V'(|D\ell|) \zeta^{p_{1}-1}) \, dx + c(\delta) \-int_{\B_{r}} \V_{|D\ell|} \left( \frac{|u-\ell|}{r}\right)\, dx\\
&\leq c \delta \-int_{\B_{r}} \omega(|u-\ell(0)|) (\V_{|D\ell|})^{*}(\V'_{|D\ell|}(|D\ell|)) \, dx + c(\delta) \-int_{\B_{r}} \V_{|D\ell|} \left( \frac{|u-\ell|}{r}\right)\, dx\\
&\leq c\delta \V(|D\ell|) \omega\left( \-int_{\B_{r}} |u-\ell(0)|\, dx \right) + c(\delta) \-int_{\B_{r}} \V_{|D\ell|} \left( \frac{|u-\ell|}{r}\right)\, dx
\end{align*}
recalling that $\omega: [0, \infty)\ri [0, 1]$ is a non-decreasing concave function, and similarly
\begin{align*}
I\!I\!I_{2}&\leq  \delta \-int_{\B_{r}} \V_{|D\ell|} (|Du-D\ell|) \zeta^{p_{1}}\, dx + c(\delta) \-int_{\B_{r}} \omega(|u-\ell(0)|) \V(|D\ell|)\, dx \\
&\leq \delta \-int_{\B_{r}} \V_{|D\ell|} (|Du-D\ell|) \zeta^{p_{1}}\, dx +  c(\delta)  \V(|D\ell|) \omega\left( \-int_{\B_{r}} |u-\ell(0)|\, dx \right). 
\end{align*}
Using $(H_{6})$ we have 
\begin{align*}
I\!V \leq \Lambda \-int_{\B_{r}} \w_{0}(x, r) \V'(|D\ell|) \left[\zeta^{p_{1}} |Du-D\ell| + p_{1}\zeta^{p_{1}-1} |u-\ell| |D\zeta|\right]\, dx = I\!V_{1}+ I\!V_{2}. 
\end{align*}

\noindent
Now taking in mind that $\w_{0}\leq 2 \Lambda$ we can see 
\begin{align*}
I\!V_{1} &\leq \delta \-int_{\B_{r}} \V_{|D\ell|} (|Du-D\ell|)\zeta^{p_{1}} \, dx + c(\delta) \-int_{\B_{r}} (\V_{|D\ell|})^* \left( \w_{0}(x, r) \V'(|D\ell|)\right)\, dx \\
&\leq \delta \-int_{\B_{r}} \V_{|D\ell|} (|Du-D\ell|)\zeta^{p_{1}} \, dx + c(\delta) \-int_{\B_{r}} \w_{0}(x, r) \V(|D\ell|)\, dx \\
&\leq  \delta \-int_{\B_{r}} \V_{|D\ell|} (|Du-D\ell|)\zeta^{p_{1}} \, dx + c {\rm W}(r) \V(|D\ell|). 
\end{align*}
Analogously, 
\begin{align*}
I\!V_{2} &\leq \delta \-int_{\B_{r}} (\V_{|D\ell|})^{*} (\w_{0}(x, r)\V'(|D\ell|)\zeta^{p_{1}-1})\, dx + c(\delta) \-int_{\B_{r}} \V_{|D\ell|} \left( \frac{|u-\ell|}{r}\right)\, dx \\
&\leq c {\rm W}(r) \V(|D\ell|) + c(\delta) \-int_{\B_{r}} \V_{|D\ell|} \left( \frac{|u-\ell|}{r}\right)\, dx. 
\end{align*}
Finally, using $(H_{7})$ and the fact that $\V'(a+t)\leq 2 \V'_{a}(a+t)$, we obtain
\begin{align*}
V &\leq L \-int_{\B_{r}} \V'(|Du|) \zeta^{p_{1}}|u-\ell|\, dx \\
&\leq L \-int_{\B_{r}} \V'(|D\ell|+ |Du-D\ell|) \zeta^{p_{1}}|u-\ell|\, dx  \\
&\leq 2L \-int_{\B_{r}} \V'_{|D\ell|}(|D\ell|+ |Du-D\ell|) \zeta^{p_{1}}|u-\ell|\, dx  \\ 
&= \frac{2L}{|\B_{r}|} \int_{\B_{r}\cap \{ |Du-D\ell|>|D\ell|\}} \V'_{|D\ell|}(|D\ell|+ |Du-D\ell|) \zeta^{p_{1}}|u-\ell|\, dx \\
&\qquad + \frac{2L}{|\B_{r}|} \int_{\B_{r}\cap \{ |Du-D\ell|\leq |D\ell|\}} \V'_{|D\ell|}(|D\ell|+ |Du-D\ell|) \zeta^{p_{1}}|u-\ell|\, dx\\
&= V_{1}+ V_{2}. 
\end{align*} 
Exploiting the $\Delta_{2}$-condition we deduce
\begin{align*}
V_{1} &\le c\int_{\B_{r}\cap \{ |Du-D\ell|>|D\ell|\}} \V'_{|D\ell|}(|Du-D\ell|) \zeta^{p_{1}}|u-\ell|\, dx\\
&\leq \delta \-int_{\B_{r}} \V_{|D\ell|} (|Du-D\ell|)\zeta^{p_{1}} \, dx + c(\delta)\-int_{\B_{r}} \V_{|D\ell|}(|u-\ell|)\, dx \\
&\leq \delta \-int_{\B_{r}} \V_{|D\ell|} (|Du-D\ell|)\zeta^{p_{1}} \, dx + c(\delta) r\-int_{\B_{r}} \V_{|D\ell|}\left(\frac{|u-\ell|}{r}\right)\, dx 
\end{align*}
and also
\begin{align*}
V_{2}& \le c\, r \-int_{\B_{r}} \V'_{|D\ell|}(|D\ell|) \zeta^{p_{1}} \frac{|u-\ell|}{r}\, dx \\
&\leq r \-int_{\B_{r}} \V_{|D\ell|}\left( \frac{|u-\ell|}{r}\right) \, dx + r\V(|D\ell|). 
\end{align*}
The terms multiplied by $\delta$ can be absorbed into the left hand side through a suitable choice of $\delta$, thus arriving to the desired estimate. Observe that $c_{cacc}$ depends on $\nu,\Lambda, L$ and $p_0,p_1$.
\end{proof}

\noindent
Choosing $\ell=u_{0}$ in Theorem \ref{Caccioppoli}, we get the following zero-order Caccioppoli inequality. 
\begin{cor}\label{cacczero}
Let $u\in W^{1, \V}(\Omega, \R^{N})$ be a weak solution to \eqref{P} under the same assumptions of Theorem \ref{mainthm}. Then, for every ball $\B_{r}(x_{0})\Subset \Omega$ and for every $u_{0}\in \R^{N}$ we have 
\begin{align}\label{Caccioppoli-zero}
\-int_{\B_{\frac{r}{2}}(x_{0})} \V(|Du|)\, dx \leq c_{cacc} \-int_{\B_{r}(x_{0})} \V \left(\frac{|u-u_{0}|}{r}\right)\, dx.  
\end{align}
\end{cor}

\noindent
From Corollary \ref{cacczero} it is possible to deduce in a standard way with the help of Gehring's Lemma the following higher integrability result (see Theorem 2.5 in \cite{CO}).

\begin{thm}\label{higherintegrability}
There exists $\gamma>1$ such that $\V(|Du|)\in L^{\gamma}_{loc}(\Omega)$ and 
\begin{align*}
\left(\-int_{\B_{\frac{r}{2}}(x_{0})} \V(|Du|)^{\gamma}\, dx \right)^{\frac{1}{\gamma}} \leq c_h\, \-int_{\B_{r}(x_{0})} \V(|Du|)\, dx
\end{align*}
for a suitable constant $c_h=c(n, N, p_{0}, p_{1}, \nu, \Lambda)>0$. 
\end{thm}

\section{Approximate harmonicity by linearization}

\noindent
For $x_{0}\in \Omega$ and $r\in (0, {\rm dist}(x_{0}, \partial \Omega))$ we define the following excess functionals
\begin{align*}
\Phi(x_{0}, r) =\-int_{\B_{r}(x_{0})} \V_{|D\ell_{x_{0}, r}|} \left( \frac{|u- \ell_{x_{0}, r}|}{r}\right)\, dx
\end{align*}
and 
\begin{align*}
\Phi^{*}(x_{0}, r)= \Phi(x_{0}, r) + \h(x_{0}, r)^{\beta_{1}} \V(|D\ell_{x_{0}, r}|), 
\end{align*}
where
\begin{align*}
\h(x_{0}, r)= {\rm W}(r)+ \omega \left( \-int_{\B_{r}(x_{0})} |u-\ell_{x_0,r}(x_{0})|\, dx\right) +r, 
\end{align*}
and $\beta_{1}= \min\left\{\left(1-\frac{1}{\gamma}\right) \left(1-\frac{1}{p_{0}}\right), \frac{\beta_{0}+1}{2}\right\}$, with $\beta_0$ defined in $(H_4)$ and $\gamma$ given in Theorem \ref{higherintegrability}. 
Then we set
\begin{align*}
\Upsilon= \sqrt{\frac{\Phi^{*}(x_{0}, r)}{\V(|D\ell_{x_{0}, r}|)}}, 
\end{align*}
and for $|D\ell_{x_{0}, r}|\neq 0$ we define
\begin{align*}
\A= \frac{\left(D_{\xi} a( \cdot, \ell_{x_{0}, r}(x_{0}), D\ell_{x_{0}, r}) \right)_{x_{0}, r}}{\V''(|D\ell_{x_{0}, r}|)}. 
\end{align*}
Note that from hypotheses $(H_{2})$ and $(H_{3})$ it follows that for any $\xi, \eta \in \R^{nN}$ 
\begin{align*}
(\A \xi, \xi)\geq \nu |\xi|^{2} \quad \mbox{ and } \quad |\A (\xi, \eta)|\leq \Lambda |\xi| |\eta|. 
\end{align*}

\begin{lem}\label{lem3.9Bog}
Suppose that 
\begin{align}\label{eq15}
\Phi(x_{0}, r)\leq \V(|D\ell_{x_{0}, r}|), 
\end{align}
then $u-\ell_{x_0,r}$ is approximately $\A$-harmonic on the ball $\B_{\frac{r}{4}}(x_{0})$ in the sense that 
\begin{align*}
\left| \-int_{\B_{\frac{r}{4}}(x_{0})} \A(D(u- \ell_{x_{0}, r}), D\zeta) \, dx \right| \leq c_0 \|D\zeta\|_{\infty} |D\ell_{x_{0}, r}| \left\{\frac{\Phi(x_{0}, r)}{\V(|D\ell_{x_{0}, r}|)} + \left[ \frac{\Phi(x_{0}, r)}{\V(|D\ell_{x_{0}, r}|)}\right]^{\frac{\beta_{0}+1}{2}} + \h(x_{0}, r)^{\beta_{1}} \right\}
\end{align*}
holds for all $\zeta \in \C^{\infty}_{c}(\B_{\frac{r}{4}}(x_{0}), \R^{N})$, where $\beta_{1}= \min\left\{\left(1-\frac{1}{\gamma}\right) \left(1-\frac{1}{p_{0}}\right), \frac{\beta_{0}+1}{2}\right\}$ and $\gamma$ is given in Theorem \ref{higherintegrability}. 
\end{lem}

\begin{proof}
Assume that $x_{0}=0$ for simplicity. Take $\zeta \in \C^{\infty}_{c}(\B_{\frac{r}{4}}, \R^{N})$ and set $v= u-\ell_{r}$. Then, 
\begin{align*}
&\-int_{\B_{\frac{r}{4}}} \left( \A Dv, D\zeta \right)\, dx \\
&= \frac{1}{\V''(|D\ell_{r}|)} \-int_{\B_{\frac{r}{4}}} \int_{0}^{1} \left( \left[(D_{\xi}a(\cdot, \ell_{r}(0), D\ell_{r}))_{r} - (D_{\xi}a(\cdot, \ell_{r}(0), D\ell_{r} +tDv))_{r}\right]Dv, D\zeta  \right)\, dx \\
&\quad + \frac{1}{\V''(|D\ell_{r}|)} \-int_{\B_{\frac{r}{4}}} \int_{0}^{1} \left( (D_{\xi}a(\cdot, \ell_{r}(0), D\ell_{r} +tDv))_{r}Dv, D\zeta  \right)\, dx \\
&=I+I\!I. 
\end{align*} 

\noindent
{\it Estimate of $I$.} First we define the sets
\begin{align*}
\X=\B_{\frac{r}{4}} \cap \left\{|Du-D\ell_{r}|\geq \frac{|D\ell_{r}|}{2}\right\} \quad \mbox{ and } \quad \Y=\B_{\frac{r}{4}} \cap \left\{|Du-D\ell_{r}|< \frac{|D\ell_{r}|}{2}\right\}, 
\end{align*}
so that we can split $I$ as follows: 
\begin{align*}
I&= \frac{1}{\V''(|D\ell_{r}|)|\B_{\frac{r}{4}}|} \int_{\X}  \int_{0}^{1} \left( \left[(D_{\xi}a(\cdot, \ell_{r}(0), D\ell_{r}))_{r} - (D_{\xi}a(\cdot, \ell_{r}(0), D\ell_{r} +tDv))_{r}\right]Dv, D\zeta  \right)\, dt \,dx \\
&+\frac{1}{\V''(|D\ell_{r}|)|\B_{\frac{r}{4}}|} \int_{\Y}  \int_{0}^{1} \left( \left[(D_{\xi}a(\cdot, \ell_{r}(0), D\ell_{r}))_{r} - (D_{\xi}a(\cdot, \ell_{r}(0), D\ell_{r} +tDv))_{r}\right]Dv, D\zeta  \right)\, dt \, dx \\
&=I_{1}+ I_{2}. 
\end{align*}
Using $(H_{3})$ we have
\begin{align*}
I_{1}&=\frac{1}{\V''(|D\ell_{r}|)} \-int_{\B_{\frac{r}{4}}}\chi_{\X}  \int_{0}^{1} \left( \left[(D_{\xi}a(\cdot, \ell_{r}(0), D\ell_{r}))_{r} - (D_{\xi}a(\cdot, \ell_{r}(0), D\ell_{r} +tDv))_{r}\right]Dv, D\zeta  \right)\, dt \, dx \\
&\leq \frac{2\Lambda}{\V''(|D\ell_{r}|)}\-int_{\B_{\frac{r}{4}}}\chi_{\X}  \int_{0}^{1} \left[\V''(|D\ell_{r}|) + \V''(|D\ell_{r} +tDv|)\right] |Dv| |D\zeta|\, dt \,dx.  
\end{align*}
We point out that from \eqref{t3} and Lemma \ref{lem2} we obtain
\begin{align*}
\int_{0}^{1} \left[\V''(|D\ell_{r}|)+ \V''(|D\ell_{r}+t Dv|) \right] \, dt &\sim \int_{0}^{1} \left[\frac{\V'(|D\ell_{r}|)}{|D\ell_{r}|}+ \frac{\V'(|D\ell_{r}+t (Du- D\ell_{r})|)}{|D\ell_{r}+ t(Du- D\ell_{r})|} \right] \, dt \nonumber\\
&\sim \frac{\V'(|D\ell_{r}|)}{|D\ell_{r}|}+ \frac{\V'(|D\ell_{r}|+|Du|)}{|D\ell_{r}|+|Du|} \nonumber \\
&\leq \frac{\V'(|D\ell_{r}|) + \V'(|D\ell_{r}|+|Du|)}{|D\ell_{r}|} \nonumber\\
&\leq 2\frac{\V'(|D\ell_{r}|+|Du|)}{|D\ell_{r}|};
\end{align*}
then, recalling that on $\X$ we have $2|Du-D\ell_{r}|\geq |D\ell_{r}|$, we can say that $\V'(|D\ell_{r}|+|Du|)\lesssim \V'(|Du-D\ell_{r}|)$, so exploiting \eqref{t3}, $\V_{a}(t)\sim \V(t)$ for $t\geq a$ and Theorem \ref{Caccioppoli} we obtain 
\begin{align*}
I_{1}&\le c \frac{1}{\V''(|D\ell_{r}|)} \-int_{\B_{\frac{r}{4}}} \chi_{\X} \frac{\V'(|D\ell_{r}|+|Du|)}{|D\ell_{r}|} |Du-D\ell_{r}|\, dx\|D\zeta\|_{\infty}\\
&\le c\frac{|D\ell_{r}|}{\V(|D\ell_{r}|)} \-int_{\B_{\frac{r}{4}}} \V_{|D\ell_{r}|}(|Du-D\ell_{r}|)\, dx \|D\zeta\|_{\infty}\\
&\le c \frac{|D\ell_{r}|}{\V(|D\ell_{r}|)} \left( \Phi(r) + \V(|D\ell_{r}|)\h(r)\right) \|D\zeta\|_{\infty}\\
&= c|D\ell_{r}| \left( \frac{\Phi(r)}{\V(|D\ell_{r}|)} + \h(r)\right)\|D\zeta\|_{\infty}. 
\end{align*}
To estimate $I_{2}$ we use $(H_{4})$, H\"older inequality, \eqref{t6}, Theorem \ref{Caccioppoli} 
\begin{align*}
I_{2}&=\frac{1}{\V''(|D\ell_{r}|)|\B_{\frac{r}{4}}|} \-int_{\B_{r}} \int_{0}^{1} \int_{\Y} \left( \left[D_{\xi}a(y, \ell_{r}(0), D\ell_{r}) - D_{\xi}a(y, \ell_{r}(0), D\ell_{r} +tDv(x))\right]Dv, D\zeta \right)\, dx\, dt\, dy\\
&\le c\frac{1}{|\B_{\frac{r}{4}}|} \int_{0}^{1} \int_{\Y} \left( \frac{t|Du-D\ell_{r}|}{|D\ell_{r}|} \right)^{\beta_{0}} |Du-D\ell_{r}|\, dx \, dt\|D\zeta\|_{\infty}\\
&\le c|D\ell_{r}| \-int_{\B_{\frac{r}{4}}} \chi_{\Y} \left( \frac{|Du-D\ell_{r}|}{|D\ell_{r}|}\right)^{\beta_{0}+1}\, dx\|D\zeta\|_{\infty} \\
&\le c |D\ell_{r}| \-int_{\B_{\frac{r}{4}}} \chi_{\Y} \left( \frac{\V'(|D\ell_{r}|)|Du-D\ell_{r}|^{2}}{\V'(|D\ell_{r}|) |D\ell_{r}|^{2}}\right)^{\frac{\beta_{0}+1}{2}}\, dx\|D\zeta\|_{\infty} \\
&\le c |D\ell_{r}| \-int_{\B_{\frac{r}{4}}} \chi_{\Y} \left[ \frac{\V'(|D\ell_{r}| + |Du-D\ell_{r}|) |Du-D\ell_{r}|^{2}}{\V(|D\ell_{r}|) (|D\ell_{r}|+ |Du-D\ell_{r}|)} \right]^{\frac{\beta_{0}+1}{2}}\, dx \|D\zeta\|_{\infty}\\
&\le c |D\ell_{r}| \-int_{\B_{\frac{r}{4}}} \chi_{\Y} \left[ \frac{\V_{|D\ell_{r}|}(|Du-D\ell_{r}|)}{\V(|D\ell_{r}|)}\right]^{\frac{\beta_{0}+1}{2}}\, dx\|D\zeta\|_{\infty} \\
&\le c |D\ell_{r}| \left( \frac{1}{\V(|D\ell_{r}|)} \-int_{\B_{\frac{r}{4}}} \V_{|D\ell_{r}|}(|Du-D\ell_{r}|) \, dx \right)^{\frac{\beta_{0}+1}{2}}\|D\zeta\|_{\infty}\\
&\le c |D\ell_{r}| c_{cacc} \left(\frac{\Phi(r)}{\V(|D\ell_{r}|)} + \h(r)\right)^{\frac{\beta_{0}+1}{2}}\|D\zeta\|_{\infty}.  
\end{align*}
{\it Estimate of $I\!I$.}
Using
\begin{align*}
\-int_{\B_{\frac{r}{4}}} a(x, u, Du)D\zeta\, dx = \-int_{\B_{\frac{r}{4}}} b(x, u, Du)\zeta\, dx \quad \mbox{ and } \quad \-int_{\B_{\frac{r}{4}}} \left((a(\cdot, \ell_{r}(0), D\ell_{r}))_{r}, D\zeta \right)\, dx =0
\end{align*}
we obtain 
\begin{align*}
I\!I&=\frac{1}{\V''(|D\ell_{r}|)} \-int_{\B_{\frac{r}{4}}} \int_{0}^{1} \left( \frac{d}{dt} (a(\cdot, \ell_{r}(0), D\ell_{r} +tDv))_{r}, D\zeta  \right)\, dx\\
&=\frac{1}{\V''(|D\ell_{r}|)} \-int_{\B_{\frac{r}{4}}} \left( (a(\cdot, \ell_{r}(0) , Du))_{r}- (a(\cdot, \ell_{r}(0), D\ell_{r}))_{r}, D\zeta \right)\, dx \\
&=\frac{1}{\V''(|D\ell_{r}|)} \-int_{\B_{\frac{r}{4}}} \left( (a(\cdot, \ell_{r}(0) , Du))_{r}- a(x, \ell_{r}(0), Du), D\zeta \right)\, dx \\
&\qquad + \frac{1}{\V''(|D\ell_{r}|)} \-int_{\B_{\frac{r}{4}}} \left( a(x, \ell_{r}(0), Du) - a(x, u, Du), D\zeta \right)\, dx \\
&\qquad + \frac{1}{\V''(|D\ell_{r}|)} \-int_{\B_{\frac{r}{4}}} b(x, u, Du)\zeta\, dx \\
&=I\!I_{1}+ I\!I_{2} + I\!I_{3}. 
\end{align*}
Combining $(H_{6})$ together with $\w_{0}(x, r)\leq 2 \Lambda$, Theorem \ref{higherintegrability}, \eqref{Caccioppoli-zero} and \eqref{t5conj} we have
\begin{align*}
I\!I_{1} &\leq \frac{\Lambda}{\V''(|D\ell_{r}|)} \-int_{\B_{\frac{r}{4}}} \w_{0}(x, r) \V'(|Du|) |D\zeta|\, dx \\
&\le \frac{\Lambda}{\V''(|D\ell_{r}|)} (\V^{*})^{-1}\circ \V^{*} \left( \-int_{\B_{\frac{r}{4}}} \w_{0}(x, r) \V'(|Du|) \, dx \right) \|D\zeta \|_{\infty} \\
&\leq c\frac{\Lambda}{\V''(|D\ell_{r}|)} (\V^{*})^{-1} \left( \-int_{\B_{\frac{r}{4}}} \V^{*} \left(  \w_{0}(x, r) \V'(|Du|) \right)\, dx  \right) \|D\zeta \|_{\infty} \\
&\le c \frac{1}{\V''(|D\ell_{r}|)} (\V^{*})^{-1} \left( \-int_{\B_{\frac{r}{4}}} \w_{0}(x, r) \V(|Du|) \, dx  \right) \|D\zeta \|_{\infty} \\
&\leq \frac{1}{\V''(|D\ell_{r}|)} (\V^{*})^{-1} \left[\left( \-int_{\B_{\frac{r}{4}}} \w_{0}(x, r)^{\frac{\gamma}{\gamma-1}}\, dx \right)^{1-\frac{1}{\gamma}}\left( \-int_{\B_{\frac{r}{4}}} \V(|Du|)^{\gamma}\, dx \right)^{\frac{1}{\gamma}} \right]\|D\zeta \|_{\infty}\\
&\le c \frac{1}{\V''(|D\ell_{r}|)} (\V^{*})^{-1} \left[ {\rm W}(r)^{1-\frac{1}{\gamma}} \-int_{\B_{\frac{r}{2}}} \V(|Du|)\, dx \right]\|D\zeta \|_{\infty}\\
&\le c\frac{1}{\V''(|D\ell_{r}|)} (\V^{*})^{-1} \left[ {\rm W}(r)^{1-\frac{1}{\gamma}} \-int_{\B_{\frac{r}{2}}} \V\left(\frac{|u-(u)_{r}|}{r}\right)\, dx \right]\|D\zeta \|_{\infty}\\
&\le c\frac{1}{\V''(|D\ell_{r}|)} \left[ {\rm W}(r)^{1-\frac{1}{\gamma}} \right]^{1-\frac{1}{p_{0}}} (\V^{*})^{-1} \left[ \-int_{\B_{\frac{r}{2}}} \V\left(\frac{|u-(u)_{r}|}{r}\right)\, dx \right]\|D\zeta \|_{\infty}. 
\end{align*}
We point out that from \eqref{eq15} we get
\begin{align*}
\-int_{\B_{r}} \V\left( \frac{|u- (u)_{r}|}{r}\right)\, dx &\leq \-int_{\B_{r}} \V\left( \frac{|u- \ell_{r}|}{r} + \frac{|\ell_{r}- (u)_{r}|}{r}\right)\, dx\\
&\leq \-int_{\B_{r}} \V\left( \frac{|u- \ell_{r}|}{r} + |D\ell_{r}|\right)\, dx\\
&\le c \-int_{\B_{r}} \V_{|D\ell_{r}|} \left( \frac{|u- \ell_{r}|}{r}\right) +c \V(|D\ell_{r}|)\, dx\\
&= c[ \Phi(r) + \V(|D\ell_{r}|)]\leq 2c \V(|D\ell_{r}|). 
\end{align*}
Hence, using $\V^{*}\left( \frac{\V(t)}{t}\right) \sim \V(t)$ and \eqref{t3} we can see 
\begin{align*}
I\!I_{1} \le c \frac{1}{\V''(|D\ell_{r}|)} {\rm W}(r)^{\bar{\beta}} (\V^{*})^{-1} (\V(|D\ell_{r}|))\|D\zeta \|_{\infty}=c {\rm W}(r)^{\bar{\beta}} |D\ell_{r}|\|D\zeta \|_{\infty}, 
\end{align*}
where $\bar{\beta}:= \left(1-\frac{1}{\gamma}\right) \left(1-\frac{1}{p_{0}}\right)$. Exploiting $(H_{5})$ and proceeding similarly to $I\!I_{1}$ we obtain 
\begin{align*}
I\!I_{2} &\leq \frac{1}{\V''(|D\ell_{r}|)}\-int_{\B_{\frac{r}{4}}} \omega (|u-\ell_{r}(0)|)\V'(|Du|)\, dx \|D\zeta\|_{\infty} \\
&\leq \frac{1}{\V''(|D\ell_{r}|)} (\V^{*})^{-1}\circ \V^{*} \left( \-int_{\B_{\frac{r}{4}}} \omega (|u-\ell_{r}(0)|)\V'(|Du|)\, dx\right)  \|D\zeta\|_{\infty} \\
&\leq \frac{1}{\V''(|D\ell_{r}|)} (\V^{*})^{-1} \left( \-int_{\B_{\frac{r}{4}}} \V^{*} \left( \omega (|u-\ell_{r}(0)|)\V'(|Du|)\right)\, dx \right) \|D\zeta\|_{\infty} \\
&\leq \frac{1}{\V''(|D\ell_{r}|)} (\V^{*})^{-1} \left[\left(\-int_{\B_{\frac{r}{4}}} \omega (|u-\ell_{r}(0)|)^{\frac{\gamma}{\gamma-1}}\, dx \right)^{1-\frac{1}{\gamma}} \left( \-int_{\B_{\frac{r}{4}}} \V(|Du|)^{\gamma}\, dx \right)^{\frac{1}{\gamma}} \right]\|D\zeta \|_{\infty}\\
&\leq \frac{1}{\V''(|D\ell_{r}|)}\left[ \omega \left( \-int_{\B_{\frac{r}{4}}} |u-\ell_{r}(0)| \, dx \right) \right]^{\bar{\beta}} (\V^{*})^{-1} (\V(|D\ell_{r}|))\|D\zeta \|_{\infty}\\
&\leq \omega \left( \-int_{\B_{\frac{r}{4}}} |u-\ell_{r}(0)| \, dx \right)^{\bar{\beta}} |D\ell_{r}|\|D\zeta \|_{\infty}. 
\end{align*}
Now we define 
\begin{align*}
\W= \B_{\frac{r}{4}} \cap \{|Du- D\ell_{r}|\geq |D\ell_{r}|\} \quad \mbox{ and } \quad \Z= \B_{\frac{r}{4}} \cap \{|Du- D\ell_{r}|< |D\ell_{r}|\}
\end{align*}
Then, from $(H_{7})$, \eqref{t6} and Theorem \ref{Caccioppoli} we see 
\begin{align*}
I\!I_{3} &\leq \frac{L}{\V''(|D\ell_{r}|)} \-int_{\B_{\frac{r}{4}}} \V'(|Du|) |\zeta|\, dx \\
&\le c \frac{r}{\V''(|D\ell_{r}|)} \-int_{\B_{\frac{r}{4}}} \V'(|Du|)\, dx\, \|D\zeta\|_{\infty}\\
&\leq c\frac{r}{\V''(|D\ell_{r}|)} \-int_{\B_{\frac{r}{4}}} \chi_{\W} \V'(|Du-D\ell_{r}|+ |D\ell_{r}|) \, dx\, \|D\zeta\|_{\infty}\\ 
&\qquad + c\frac{r}{\V''(|D\ell_{r}|)} \-int_{\B_{\frac{r}{4}}} \chi_{\Z} \V'(|Du-D\ell_{r}|+ |D\ell_{r}|) \, dx\, \|D\zeta\|_{\infty}\\  
&\le c \frac{r}{\V''(|D\ell_{r}|)} \-int_{\B_{\frac{r}{4}}} \V'_{|D\ell_{r}|}(|Du-D\ell_{r}|) \frac{|Du-D\ell_{r}|}{|D\ell_{r}|}  \, dx \,\|D\zeta\|_{\infty}\\ 
&\qquad +c \frac{r}{\V''(|D\ell_{r}|)} \-int_{\B_{\frac{r}{4}}} \V'(|D\ell_{r}|)\, dx \,\|D\zeta\|_{\infty}\\ 
&\le c\frac{|D\ell_{r}|}{\V(|D\ell_{r}|)}\-int_{\B_{\frac{r}{4}}} \V_{|D\ell_{r}|}(|Du-D\ell_{r}|)  \, dx\, \|D\zeta\|_{\infty} + r|D\ell_{r}| \|D\zeta\|_{\infty}\\ 
&\le c \left\{ |D\ell_{r}| \left[ \frac{\Phi(r)}{\V(|D\ell_{r}|)} + \h(r)\right] + r|D\ell_{r}|\right\}\|D\zeta\|_{\infty}. 
\end{align*}
In conclusion we get
\begin{align*}
&\left| \-int_{\B_{\frac{r}{4}}} \A (Du-D\ell_{r}, D\zeta)\,dx \right| \\
&\leq c \|D\zeta\|_{\infty} |D\ell_{r}| \left\{ \left[ \frac{\Phi(r)}{\V(|D\ell_{r}|)} + \h(r)\right]^{\frac{\beta_{0}+1}{2}} + \left[ \frac{\Phi(r)}{\V(|D\ell_{r}|)} + \h(r)\right]+ {\rm W}(r)^{\bar{\beta}} \right. \\
&\qquad \qquad \qquad \qquad \qquad \left.+ \omega\left( \int_{\B_{r}} |u-\ell_{r}(0)|\, dx \right)^{\bar{\beta}}+r \right\}\\
&\le c_0\|D\zeta\|_{\infty} |D\ell_{r}| \left\{\frac{\Phi(r)}{\V(|D\ell_{r}|)} + \left[ \frac{\Phi(r)}{\V(|D\ell_{r}|)}\right]^{\frac{\beta_{0}+1}{2}} + \h(r)^{\beta_{1}} \right\}
\end{align*}
where $\beta_{1}= \min\{\bar{\beta}, \frac{\beta_{0}+1}{2}\}$. 
\end{proof}

\noindent The following lemma ensures that any solution to  \eqref{P} is approximately $\V$-harmonic. This will be the starting point for the application of the $\V$-harmonic approximation Theorem \ref{phiappr}.

\smallskip

\noindent From the assumption \eqref{nearzero} we have: for any $\delta>0$ there exists $\eta(\delta)>0$ such that if $|\xi|< \eta(\delta)$ then 
\begin{align}\label{DC}
\left| a(x, u, Du)- \frac{\xi}{|\xi|} \V'(|\xi|)\right| \leq \delta \V'(|\xi|) \quad \forall (x, u) \in \Omega \times \R^{N}. 
\end{align}
We define the following excess functional:
\begin{align*}
\Psi_{1}(x_{0}, r)= \-int_{\B_{r}(x_{0})} \V\left(\frac{|u-(u)_{x_{0}, r}|}{r}\right)\, dx. 
\end{align*}

\begin{lem}\label{lem1DC}
For any $\delta>0$ and for $\eta(\delta)>0$ given by \eqref{DC}, the inequality 
\begin{align*}
&\left|\-int_{\B_{\frac{r}{4}}(x_{0})} \left( \V'(|Du|)\frac{Du}{|Du|}, D\zeta \right) \, dx\right|\\
& \leq c_{1} \left[ \delta + \frac{1}{\eta(\delta)} \V^{-1} \left(\Psi_{1}\left(x_{0}, \frac{r}{2}\right)\right)+ {\rm W}(r)^{\frac{\gamma-1}{\gamma}} + \omega \left(\-int_{\B_{r}(x_{0})} |u-(u)_{x_{0}, r}|\, dx \right)^{\frac{\gamma-1}{\gamma}}+r\right]\times\\
&\qquad \times \left[\-int_{\B_{\frac{r}{2}}(x_{0})} \V(|Du|)\, dx + \V(\|D\zeta\|_{\infty})\right]. 
\end{align*}
holds, for every $\zeta \in \C^{\infty}_{c}(\B_{\frac{r}{4}}(x_{0}), \R^{N})$ and $c_1=c_1(n, N, \Lambda, \nu,p_0,p_1)>0$. 
\end{lem}

\begin{proof}
Again we suppose $x_{0}=0$ and we let $\zeta \in \C^{\infty}_{c}(\B_{\frac{r}{4}}, \R^{N})$. Then 
\begin{align*}
\-int_{\B_{\frac{r}{4}}} \left( \V'(|Du|)\frac{Du}{|Du|}, D\zeta \right) \, dx &\leq\-int_{\B_{\frac{r}{4}}} \left( \V'(|Du|)\frac{Du}{|Du|}- \left( a(\cdot, \ell_{r}(0), Du) \right)_{r}, D\zeta \right) \, dx \\
&\quad + \-int_{\B_{\frac{r}{4}}} ( \left( a(\cdot, \ell_{r}(0), Du) \right)_{r}- a(x, \ell_{r}(0), Du), D\zeta ) \, dx \\
&\quad + \-int_{\B_{\frac{r}{4}}} ( a(x, \ell_{r}(0), Du) - a(x, u, Du), D\zeta ) \, dx \\
&\quad + \-int_{\B_{\frac{r}{4}}} b(x, u, Du) \zeta \, dx =I+I\!I+I\!I\!I+I\!V. 
\end{align*}
{\it Estimate of $I$.} First we define the sets 
\begin{align*}
\X= \B_{\frac{r}{4}}\cap \left\{ |Du|<\eta(\delta) \right\} \quad \mbox{ and } \quad \Y= \B_{\frac{r}{4}} \cap \left\{ |Du|\geq \eta(\delta)\right\}, 
\end{align*}
with $\eta(\delta)>0$ given by \eqref{DC}, so that $I$ can be splitted as 
\begin{align*}
|I|&\leq \-int_{\B_{\frac{r}{4}}} \left[ \-int_{\B_{r}} \left| \V'(|Du|)\frac{Du}{|Du|}- a(y, \ell_{r}(0), Du) \right|\, dy \, \|D\zeta \|_{\infty} \right]\, dx \\
&=\-int_{\B_{r}} \-int_{\B_{\frac{r}{4}}} \chi_{\X} \left| \V'(|Du|)\frac{Du}{|Du|}- a(y, \ell_{r}(0), Du) \right| \, \|D\zeta \|_{\infty} dxdy \\
&\quad +\-int_{\B_{r}} \-int_{\B_{\frac{r}{4}}} \chi_{\Y} \left| \V'(|Du|)\frac{Du}{|Du|}- a(y, \ell_{r}(0), Du) \right| \, \|D\zeta \|_{\infty} dxdy \\
&=I_{1}+ I_{2}. 
\end{align*}
Using \eqref{DC}, Young's inequality and \eqref{t3}, we get
\begin{align*}
I_{1}&\leq \delta \-int_{\B_{\frac{r}{4}}} \chi_{\X} \V'(|Du|)\, \|D\zeta \|_{\infty}  dx\\
&\leq \delta \left[ \-int_{\B_{\frac{r}{4}}} \V^{*} \left( \V'(|Du|) \right) + \V (\|D\zeta\|_{\infty})\, dx \right]\\
&\leq c\,\delta \left[ \-int_{\B_{\frac{r}{4}}} \V(|Du|) + \V (\|D\zeta\|_{\infty})\, dx \right]. 
\end{align*}
Moreover, from $(H_1)$, \eqref{t3}, 
Theorem \ref{cacczero}, the fact that $(\V^{*})^{-1}(t) \V^{-1}(t)\sim t$, Young's inequality and finally Poincar\`e's inequality, we can see 
\begin{align*}
I_{2}&\leq c \-int_{\B_{\frac{r}{4}}} \chi_{\Y} \V'(|Du|)\|D\zeta \|_{\infty}\, dx \leq \frac{c}{\eta(\delta)} \-int_{\B_{\frac{r}{4}}} \V'(|Du|) |Du|\|D\zeta \|_{\infty}\, dx \\ 
&\leq \frac{c}{\eta(\delta)} \-int_{\B_{\frac{r}{4}}} \V(|Du|) \|D\zeta \|_{\infty}\, dx\leq \frac{c}{\eta(\delta)}\-int_{\B_{\frac{r}{2}}} \V\left(\frac{|u-(u)_{\frac{r}{2}}|}{\frac{r}{2}}\right) \|D\zeta\|_{\infty} \, dx \\
&\leq \frac{c}{\eta(\delta)} \Psi_{1} \left(\frac{r}{2} \right) \|D\zeta\|_{\infty}\leq \frac{c}{\eta(\delta)} (\V^{*})^{-1}\left(\Psi_{1} \left(\frac{r}{2} \right)\right) \V^{-1}\left(\Psi_{1} \left(\frac{r}{2} \right)\right)\|D\zeta\|_{\infty}\\
&\leq \frac{c}{\eta(\delta)}\V^{-1}\left(\Psi_{1} \left(\frac{r}{2} \right)\right) \left[ \V^{*} \circ (\V^{*})^{-1} \left( \Psi_{1}\left(\frac{r}{2}\right)\right) + \V(\|D\zeta\|_{\infty})\right]\\
&\leq \frac{c}{\eta(\delta)}\V^{-1}\left(\Psi_{1} \left(\frac{r}{2} \right)\right) \left[ \Psi_{1}\left(\frac{r}{2}\right)+ \V(\|D\zeta\|_{\infty})\right]\\
&\leq \frac{c}{\eta(\delta)}\V^{-1}\left(\Psi_{1} \left(\frac{r}{2} \right)\right) \left[\-int_{\B_{\frac{r}{2}}} \V(|Du|)\, dx + \V(\|D\zeta\|_{\infty})\right]. 
\end{align*}
{\it Estimate of $I\!I$.} Using $(H_{6})$, Young's inequality, Theorem \ref{higherintegrability} and recalling that $\w_{0}(x, r)\leq 2\Lambda$ we can infer that
\begin{align*}
I\!I &\leq \Lambda \-int_{\B_{\frac{r}{4}}} \w_{0}(x, r) \V'(|Du|) |D\zeta|\, dx \\
&\leq \Lambda \-int_{\B_{\frac{r}{4}}} \w_{0}(x, r) \V^{*}(\V'(|Du|))\, dx + \Lambda \-int_{\B_{\frac{r}{4}}} \w_{0}(x, r) \V(\|D\zeta\|_{\infty})\, dx \\
&\leq c \-int_{\B_{\frac{r}{4}}} \w_{0}(x, r) \V(|Du|)\, dx + \Lambda {\rm W}(r) \V(\|D\zeta\|_{\infty}) \\
&\leq c \left( \-int_{\B_{\frac{r}{4}}} \w_{0}(x, r)^{\frac{\gamma}{\gamma-1}} \, dx \right)^{1-\frac{1}{\gamma}} \left( \-int_{\B_{\frac{r}{4}}} \V(|Du|)^{\gamma}\, dx \right)^{\frac{1}{\gamma}}+ \Lambda {\rm W}(r) \V(\|D\zeta\|_{\infty}) \\
&\leq c \left[{\rm W}(r)\right]^{\frac{\gamma-1}{\gamma}} \-int_{\B_{\frac{r}{2}}} \V(|Du|)\, dx + \Lambda {\rm W}(r) \V(\|D\zeta\|_{\infty}). 
\end{align*}
{\it Estimate of $I\!I\!I$.} In a similar way, exploiting $(H_{5})$, Young's inequality, Theorem \ref{higherintegrability} we obtain
\begin{align*}
|I\!I\!I|&\leq \Lambda \-int_{\B_{\frac{r}{4}}} \omega\left( |u-\ell_{r}(0)|\right)\V'(|Du|) \|D\zeta\|_{\infty} \, dx \\
&\leq c \-int_{\B_{\frac{r}{4}}} \omega\left( |u-\ell_{r}(0)|\right)\V(|Du|)\, dx + \Lambda \-int_{\B_{\frac{r}{4}}} \omega\left( |u-\ell_{r}(0)|\right)\V(\|D\zeta\|_{\infty}) \, dx \\
&\leq c \left( \-int_{\B_{\frac{r}{4}}} \omega\left( |u-\ell_{r}(0)|\right)^{\frac{\gamma}{\gamma-1}}\, dx \right)^{1-\frac{1}{\gamma}}\-int_{\B_{\frac{r}{2}}} \V(|Du|)\, dx + \Lambda \omega \left( \-int_{\B_{\frac{r}{4}}} |u-\ell_{r}(0)|\, dx \right)\V(\|D\zeta\|_{\infty})\\
&\leq c\,\omega \left( \-int_{\B_{\frac{r}{4}}} |u-\ell_{r}(0)|\, dx \right)^{\frac{\gamma-1}{\gamma}} \-int_{\B_{\frac{r}{2}}} \V(|Du|)\, dx+ \Lambda \omega \left( \-int_{\B_{\frac{r}{4}}} |u-\ell_{r}(0)|\, dx \right)\V(\|D\zeta\|_{\infty}). 
\end{align*}
{\it Estimate of $I\!V$.} Using $(H_{7})$ and Young's inequality, we get
\begin{align*}
|I\!V|\leq L \-int_{\B_{\frac{r}{4}}} \V'(|Du|) r \|D\zeta\|_{\infty} \, dx \leq c\,r \left[ \-int_{\B_{\frac{r}{2}}} \V(|Du|)\, dx + \V(\|D\zeta\|_{\infty})\right].
\end{align*}
Combining the previous estimates we obtain the desired result. 
\end{proof}

\section{Decay Estimates}

\noindent Now we are in the position to establish the excess improvements. We start with the non-degenerate regime. The strategy of the proof is to approximate our solution by $\mathcal{A}$-harmonic functions.

\begin{lem}\label{excessdecay}
Assume that for $\tau\in (0, \frac{1}{8}]$ and for some ball $\B_{r}(x_{0})\subset \Omega$ with $r\leq \rho_{0}$ and $|D\ell_{x_{0}, r}|\neq 0$ the following assumptions are satisfied
\begin{align}\label{eq23}
\frac{\Phi(x_{0}, r)}{\V(|D\ell_{x_{0}, r}|)}\leq \delta_{1} \quad \mbox{ and } \quad \h(r)^{\beta_{1}}\leq \delta_{2}, 
\end{align}
where $\delta_{1}, \delta_{2}\in (0, 1)$ are such that 
\begin{align*}
&\delta_{1}= \min \left\{ \left( \frac{\tau^{n+1}}{8(n+2)}\right)^{p_{1}}, \left( \frac{\delta_{0}}{4}\right)^{\frac{2}{\beta_{0}}}\right\}, \\
&\delta_{2} = \left(\frac{\delta_{0}}{2}\right)^{2}, 
\end{align*}
with $\delta_{0}$ given in Theorem \ref{approxdlsv} and $\beta_{1}$  in Lemma \ref{lem3.9Bog}. 
Then, 
\begin{align*}
\Phi(x_{0}, \tau r)\leq c_2 \tau^{2} \Phi^{*}(x_{0}, r),
\end{align*}
where $c_2$ depends on $\nu, \Lambda, L, p_0, p_1$.
\end{lem}

\begin{proof}
Assume that $x_{0}=0$. 
Using the definition of $\Phi^{*}$ and assumptions \eqref{eq23} we get
\begin{align}\label{eq24*}
\Upsilon^{2}=\frac{\Phi^{*}(r)}{\V(|D\ell_{r}|)} = \frac{\Phi(r)}{\V(|D\ell_{r}|)} + \h(r)^{\beta_{1}}\leq \delta_{1}+ \delta_{2}\leq1. 
\end{align}
Setting 
\begin{align*}
w= \frac{u- \ell_{r}}{c_0|D\ell_{r}|},
\end{align*}
where $c_0$ is the constant of Lemma \ref{lem3.9Bog}, from which we deduce 
\begin{align}\label{eq24}
\-int_{\B_{\frac{r}{4}}} \A(Dw, D\eta) \, dx  \leq \|D\eta\|_{\infty}\, \Upsilon\, \left[ \left(\frac{\Phi(r)}{\V(|D\ell_{r}|)}\right)^{\frac{1}{2}} + \left(\frac{\Phi(r)}{\V(|D\ell_{r}|)}\right)^{\frac{\beta_{0}}{2}}+ \h(r)^{\frac{\beta_{1}}{2}}\right], 
\end{align}
and exploiting \eqref{eq23}, together with the choice of $\delta_1$ and $\delta_2$, we obtain that $w$ is approximately $\A$-harmonic in the sense that
\begin{align*}
\-int_{\B_{\frac{r}{4}}} \A(Dw, D\eta)\, dx \leq \|D\eta\|_{\infty} \Upsilon \, \left[\delta_{1}^{\frac{1}{2}} + \delta_{1}^{\frac{\beta_{0}}{2}} +\delta_{2}^{\frac{1}{2}} \right] \leq \|D\eta\|_{\infty} \Upsilon \, \delta_{0}.
\end{align*}
Then Theorem \ref{approxdlsv} and Corollary \ref{coro} applied to the $N$-function $t\to \V_ {|D\ell_{r}|}(|D\ell_{r}| t)$ (see Remark \ref{phiap0p1}) give the existence of  an $\A$-harmonic function $h\in \C^{\infty}(\B_{\frac{r}{4}}, \R^{N})$ such that 
\begin{align*}
\-int_{\B_{\frac{r}{4}}} \V_{|D\ell_{r}|} \left(|D\ell_{r}| \left|\frac{w-h}{\frac{r}{4}}\right|\right) \, dx \leq c \e^{s} \left[ \-int_{\B_{\frac{r}{4}}} \V_{|D\ell_{r}|} (|D\ell_{r}| |Dw|) \, dx + \V_{|D\ell_{r}|}(|D\ell_{r}|\Upsilon) \right]. 
\end{align*}
Moreover, from Lemma 5 in \cite{DGK}, it follows that $h$ verifies
\begin{equation}\label{stimeh}
r \sup_{\B_{\frac{r}{8}}} |D^{2}h| + \sup_{\B_{\frac{r}{8}}} |Dh| \leq c \-int_{\B_{\frac{r}{4}}} |Dh|\, dx 
\end{equation}
with $c= c(\A, n, N)$. 
Now we note that using $\varphi_{a}(sa)\leq c s^{2} \varphi_{a}(a)$ for all $s\in [0,1]$, and taking in mind \eqref{eq24*} we have 
\begin{align*}
\V_{|D\ell_{r}|}(|D\ell_{r}|\Upsilon) \leq c \Upsilon^{2} \V_{|D\ell_{r}|}(|D\ell_{r}|) = c\Upsilon^{2} \V(|D\ell_{r}|). 
\end{align*}
From Theorem \ref{Caccioppoli}, \eqref{t1} and the definition of $\Phi^{*}$ we have
\begin{equation}\label{eq29}
\begin{split}
&\-int_{\B_{\frac{r}{4}}} \V_{|D\ell_{r}|} (|D\ell_{r}| |Dw|) \, dx\\
&\quad =\-int_{\B_{\frac{r}{4}}} \V_{|D\ell_{r}|} \left( \frac{|Du-D\ell_{r}|}{c_0}\right)\, dx \leq c \-int_{\B_{\frac{r}{4}}} \V_{|D\ell_{r}|} \left(|Du-D\ell_{r}|\right)\, dx \\
&\quad  \leq c_{cacc} \left\{ \-int_{\B_{\frac{r}{2}}} \V_{|D\ell_{r}|} \left(\frac{|u-\ell_{r}|}{\frac{r}{2}}\right) \, dx + \V(|D\ell_{r}|) \left[ {\rm W}(r) + \omega \left( \-int_{\B_{\frac{r}{2}}} |u-(u)_{r}|\, dx\right)+ r\right] \right\}\\
&\quad \leq c \left\{ \-int_{\B_{r}} \V_{|D\ell_{r}|} \left(\frac{|u-\ell_{r}|}{r}\right) \, dx + \V(|D\ell_{r}|) \h(r)\right\} \\
&\quad \leq c \left[\Phi(r)+ \V(|D\ell_{r}|) \h(r)^{\beta_{1}}\right] \\
&\quad =c \frac{\Phi^{*}(r)}{\V(|D\ell_{r}|)} \V(|D\ell_{r}|)= c \Upsilon^{2} \V(|D\ell_{r}|).\\
\end{split}
\end{equation}
Hence
\begin{align}\label{eq25}
\-int_{\B_{\frac{r}{4}}} \V_{|D\ell_{r}|} \left(|D\ell_{r}| \left|\frac{w-h}{\frac{r}{4}}\right|\right) \, dx \leq c \e^{s} \Upsilon^{2} \V(|D\ell_{r}|). 
\end{align}

\noindent
Then, applying Theorem 18 in \cite{DLSV} to the $N$-function $t\to \V_{|D\ell_{r}|}(|D\ell_{r}|t)$ and using (\ref{eq29}), we can infer 
\begin{align*}
\-int_{\B_{\frac{r}{4}}} \V_{|D\ell_{r}|}(|D\ell_{r}| |Dh|)\, dx \le c \-int_{\B_{\frac{r}{4}}} \V_{|D\ell_{r}|}(|Du-D\ell_{r}|)\, dx \le c \Upsilon^{2}  \V(|D\ell_{r}|). 
\end{align*}
Furthermore, we note that by \eqref{t6} for any $t\leq 1$ we have 
\begin{align*}
\frac{\V_{|D\ell_{r}|}(|D\ell_{r}|t)}{\V(|D\ell_{r}|)} \sim \frac{t^{2}}{(1+t)^{2}} \frac{\V(|D\ell_{r}|(1+t))}{\V(|D\ell_{r}|)} \geq \frac{t^{2}}{(1+t)^{2}} \geq \frac{t^{2}}{4}. 
\end{align*}
Hence, naming $\psi(t)= \V_{|D\ell_{r}|}(|D\ell_{r}|t)$ and using \eqref{eq24*} and the above inequalities we get
\begin{align*}
\-int_{\B_{\frac{r}{4}}} |Dh|\, dx &= \psi^{-1} \circ \psi \left( \-int_{\B_{\frac{r}{4}}} |Dh|\, dx \right) \leq \psi^{-1}\left( \-int_{\B_{\frac{r}{4}}} \psi(|Dh|)\, dx \right)\\
&= \psi^{-1}\left( \-int_{\B_{\frac{r}{4}}} \V_{|D\ell_{r}|}(|D\ell_{r}| |Dh|) \, dx \right) \le c\, \psi^{-1} (\Upsilon^{2}  \V(|D\ell_{r}|)) \\
&\le c\, \psi^{-1}(\V_{|D\ell_{r}|} (|D\ell_{r}|\Upsilon)) \le c\, \psi^{-1}(\psi(\Upsilon))=c \Upsilon. 
\end{align*}
Hence, fixing $\tau \in \left(0, \frac{1}{8}\right]$ and defining $\ell^{(h)}(x)= h(0)+ Dh(0) x$, by (\ref{stimeh}) we have 
\begin{align*}
\sup_{\B_{\tau r}} |h(x) - \ell^{(h)}(x)|
\leq c \, \tau^{2} r\, \-int_{\B_{\frac{r}{4}}} |Dh|\, dx \leq c \tau^{2}r \Upsilon, 
\end{align*}
from which, recalling that $\varphi_{a}(sa)\leq c s^{2} \varphi_{a}(a)$ for all $s\in [0,1]$, we see  
\begin{align*}
\-int_{\B_{\tau r}} \V_{|D\ell_{r}|} \left( |D\ell_{r}| \left|\frac{h-\ell^{(h)}}{\tau r}\right|\right)\, dx \leq c\-int_{\B_{\tau r}} \V_{|D\ell_{r}|} \left( |D\ell_{r}| \Upsilon \tau \right)\, dx \leq c \Upsilon^{2} \tau^{2} \V(|D\ell_{r}|). 
\end{align*}
Therefore, from the above inequality, \eqref{t1} and \eqref{eq25} we have 
\begin{align*}
&\-int_{\B_{\tau r}} \V_{|D\ell_{r}|} \left( |D\ell_{r}| \left|\frac{w- \ell^{(h)}}{\tau r}\right|\right)\, dx \\
&\quad \leq c \-int_{\B_{\tau r}} \V_{|D\ell_{r}|} \left( |D\ell_{r}| \left|\frac{w-h}{\tau r}\right|\right)\, dx + c \-int_{\B_{\tau r}} \V_{|D\ell_{r}|} \left( |D\ell_{r}| \left|\frac{h- \ell^{(h)}}{\tau r}\right|\right)\, dx\\
&\quad \leq \frac{c}{(4\tau)^{n}} \-int_{\B_{\frac{r}{4}}} \V_{|D\ell_{r}|} \left( |D\ell_{r}| \left|\frac{w-h}{\frac{r}{4}}\right| \frac{\frac{r}{4}}{\tau r}\right)\, dx+ c \V(|D\ell_{r}|) \Upsilon^{2}\tau^{2}\\
&\quad \leq \frac{c}{(4\tau)^{n}} \frac{1}{(4\tau)^{p_{1}}} \-int_{\B_{\frac{r}{4}}} \V_{|D\ell_{r}|} \left( |D\ell_{r}| \left|\frac{w-h}{\frac{r}{4}}\right|\right)\, dx+ c \V(|D\ell_{r}|) \Upsilon^{2}\tau^{2}\\
&\quad \leq c \tau^{-n-p_{1}} \e^{s} \Upsilon^{2} \V(|D\ell_{r}|)+ c \V(|D\ell_{r}|) \Upsilon^{2}\tau^{2}. 
\end{align*}
Taking $\e^{s}\leq \tau^{n+p_{1}+2}$ we get
\begin{align*}
\-int_{\B_{\tau r}} \V_{|D\ell_{r}|} \left( |D\ell_{r}| \left|\frac{w- \ell^{(h)}}{\tau r}\right|\right)\, dx \leq c \V(|D\ell_{r}|)\Upsilon^{2} \tau^{2} = c\, \Phi^{*}(r)\tau^{2}.
\end{align*}
On the other hand, by Lemma \ref{minaffine},
\begin{align*}
\-int_{\B_{\tau r}} \V_{|D\ell_{r}|} \left( |D\ell_{r}| \left|\frac{w- \ell^{(h)}}{\tau r}\right|\right)\, dx&= \-int_{\B_{\tau r}} \V_{|D\ell_{r}|} \left( |D\ell_{r}| \left| \frac{u-\ell_{r}}{c_0 |D\ell_{r}|} - \ell^{(h)} \right|\frac{1}{\tau r}\right)\, dx \\
&= \-int_{\B_{\tau r}} \V_{|D\ell_{r}|} \left(\frac{|u-\ell_{r} - \ell^{(h)} c_0 |D\ell_{r}||}{c_0\tau r}\right)\, dx \\
&\geq c \-int_{\B_{\tau r}} \V_{|D\ell_{r}|} \left(\frac{|u-\ell_{r} - \ell^{(h)} c_0 |D\ell_{r}||}{\tau r}\right)\, dx \\
&\geq  c \-int_{\B_{\tau r}} \V_{|D\ell_{r}|} \left(\frac{|u- \ell_{\tau r}|}{\tau r}\right)\, dx.   
\end{align*}
Therefore 
\begin{align*}
\-int_{\B_{\tau r}} \V_{|D\ell_{r}|} \left(\frac{|u- \ell_{\tau r}|}{\tau r}\right)\, dx \leq c\, \Phi^{*}(r)\tau^{2}.
\end{align*}
Let us point out that $|D\ell_{r}|\sim |D\ell_{\tau r}|$. Indeed, from \eqref{minimizza}, \eqref{eq23} and \eqref{t4} we obtain
\begin{align*}
|D\ell_{r}- D\ell_{\tau r}| &\leq (n+2) \-int_{\B_{\tau r}} \left|\frac{u- \ell_{r}}{\tau r}\right|\, dx \leq \frac{n+2}{\tau^{n+1}}\-int_{\B_{r}} \left|\frac{u- \ell_{r}}{r}\right|\, dx\\
&= \frac{n+2}{\tau^{n+1}} (\V_{|D\ell_{r}|})^{-1} \circ \V_{|D\ell_{r}|} \left( \-int_{\B_{r}} \left|\frac{u- \ell_{r}}{r}\right|\, dx \right) \\
&\leq \frac{n+2}{\tau^{n+1}} (\V_{|D\ell_{r}|})^{-1} \left( \-int_{\B_{r}} \V_{|D\ell_{r}|} \left( \left|\frac{u- \ell_{r}}{r}\right|\right)\, dx \right)\\
&= \frac{n+2}{\tau^{n+1}} (\V_{|D\ell_{r}|})^{-1} \left( \Phi(r)\right) \\
&\leq \frac{n+2}{\tau^{n+1}} (\V_{|D\ell_{r}|})^{-1} (\delta_{1} \V(|D\ell_{r}|)) \\
&\leq \frac{n+2}{\tau^{n+1}} \delta_{1}^{\frac{1}{p_{1}}} (\V_{|D\ell_{r}|})^{-1} (4\V_{|D\ell_{r}|}(|D\ell_{r}|))\\
&\le 4\frac{n+2}{\tau^{n+1}} \delta_{1}^{\frac{1}{p_{1}}} |D\ell_{r}|\leq \frac{1}{2} |D\ell_{r}|
\end{align*}
considering that $\delta_{1} \leq \left(\frac{\tau^{n+1}}{2(n+2)}\right)^{p_{1}}$. 
Hence
\begin{align*}
|D\ell_{\tau r}|\leq |D\ell_{\tau r} - D\ell_{r}|+ |D\ell_{r}|\leq \frac{3}{2} |D\ell_{r}|. 
\end{align*}
Moreover
\begin{align*}
|D\ell_{r}|\leq |D\ell_{r}- D\ell_{\tau r}|+ |D\ell_{\tau r}|\leq \frac{1}{2}|D\ell_{r}| + |D\ell_{\tau r}| 
\end{align*}
and so $\frac{1}{2}|D\ell_{r}|\leq |D\ell_{\tau r}|$, from which $|D\ell_{r}|\sim |D\ell_{\tau r}|$ and $\V_{|D\ell_{\tau r}|}(\cdot)\sim \V_{|D\ell_{r}|}(\cdot)$. Combining the previous estimates we can conclude that 
\begin{align*}
\-int_{\B_{\tau r}} \V_{|D\ell_{\tau r}|} \left( \frac{|u-\ell_{\tau r}|}{\tau r}\right) \, dx \leq c\, \Phi^{*}(r)\tau^{2}, 
\end{align*}
that is 
\begin{align*}
\Phi(\tau r)\leq c_2\, \Phi^{*}(r)\tau^{2}. 
\end{align*}
\end{proof}

\noindent
In what follows we will iterate the excess-decay estimate from Lemma \ref{excessdecay}. We define 
\begin{align*}
\Psi_{\alpha}(x_{0}, r)= r^{p_{0}(1-\alpha)} \-int_{\B_{r}(x_{0})} \V\left(\frac{|u-(u)_{x_{0}, r}|}{r} \right) \, dx \quad \mbox{ for } \alpha \in (0,1).
\end{align*}
\begin{lem}\label{iterazione} 
Let $\alpha \in (0, 1)$; then, there exist constants $\mu_{*},\kappa_{*}\in (0, 1)$, $r_{*}\in (0, \rho_{0})$ and $\tau \in (0, \frac{1}{8}]$ such that the conditions 
\begin{align}\label{assitera}
\frac{\Phi(x_{0}, r)}{\V(|D\ell_{x_{0}, r}|)}< \mu_{*} \quad \mbox{ and } \quad  \Psi_{\alpha}(x_{0}, r)<\kappa_{*}
\end{align}
for $\B_{r}(x_{0})\subset \Omega$ with $r\leq r_{*}$, imply 
\begin{align}\label{thesiitera}
\frac{\Phi(x_{0}, \tau^{j} r)}{\V(|D\ell_{x_{0}, \tau^{j}  r}|)}< \mu_{*} \quad \mbox{ and } \quad  \Psi_{\alpha}(x_{0},\tau^{j}  r)<\kappa_{*}
\end{align}
for every $j\in \N_{0}$.
\end{lem} 

\begin{proof}
Assume $x_{0}=0$ and let us choose the constants. First we let
$$
\tau \leq \min \left\{ \frac{1}{8}, (2^{p_{1}+1}c_2)^{-\frac{1}{2}}, ( 2(n+2)^{p_{1}} \bar{c})^{-\frac{1}{p_{0}(1-\alpha)}}\right\}
$$
with $c_2$ determined in Lemma \ref{excessdecay} and $\bar{c}$ determined next in \eqref{cbar}. This fixes the constant $\delta_0$ of Theorem \ref{approxdlsv} (since $\e$ has been chosen such that $\e^{s}\leq \tau^{n+p_{1}+2}$). Next we choose
$$
\mu_{*}\leq \min \left\{ \left( \frac{\tau^{n+1}}{8(n+2)}\right)^{p_{1}}, \left(\frac{\delta_{0}}{4}\right)^{\frac{2}{\beta_0}},\frac{\tau^{p_{1}+n-p_{0}(1-\alpha)}}{2\bar{c} (n+2)^{p_{1}}}\right\}. 
$$
and $\kappa_*$ and $r_*$ so small that
\begin{align*}
\left[{\rm W}(r_{*}) +r_*+ \omega \left( r_{*}+ {\tilde c}r_{*}^{\alpha} \kappa_*^{\frac{1}{p_0}}\right) \right]^{\beta_{1}}\leq \min\left\{\left(\frac{\delta_{0}}{2}\right)^{2},\mu_*\right \},
\end{align*}
with ${\tilde c}=\frac{1}{\V(1)^{\frac{1}{p_0}}}$.

\noindent Now we prove the lemma by induction on $j$. If $j=0$ then \eqref{thesiitera} is implied by \eqref{assitera}. Assume that \eqref{thesiitera} is verified by $j\in \N_{0}$ and let us prove it for $j+1$. 
We first notice that H\"older inequality together with \eqref{t1} produces 
\begin{align*}
\-int_{\B_{r}} |u-(u)_{r}|\, dx &= \frac{1}{|\B_{r}|} \int_{\B_{r}\cap \left\{\frac{|u-(u)_{r}|}{r}\leq 1\right\}} |u-(u)_{r}|\, dx+ \frac{1}{|\B_{r}|} \int_{\B_{r}\cap \left\{\frac{|u-(u)_{r}|}{r}> 1\right\}} |u-(u)_{r}|\, dx \\
&\leq r + \frac{r}{|\B_{r}|} \int_{\B_{r}\cap \left\{\frac{|u-(u)_{r}|}{r}> 1\right\}} \frac{|u-(u)_{r}|}{r}\, dx \\
&\leq r+ \frac{r}{|\B_{r}|} \left( \int_{\B_{r}\cap \left\{\frac{|u-(u)_{r}|}{r}> 1\right\}} \left|\frac{u-(u)_{r}}{r}\right|^{p_{0}}\, dx\right)^{\frac{1}{p_{0}}} |\B_{r}|^{1-\frac{1}{p_{0}}}\\
&= r+ \frac{r}{|\B_{r}|^{\frac{1}{p_{0}}}} \left( \int_{\B_{r}\cap \left\{\frac{|u-(u)_{r}|}{r}> 1\right\}} \left|\frac{u-(u)_{r}}{r}\right|^{p_{0}} \V(1) \, dx\right)^{\frac{1}{p_{0}}} \frac{1}{\V(1)^{\frac{1}{p_{0}}}}\\
&\leq r+ \frac{r}{\V(1)^{\frac{1}{p_{0}}}} \left( \-int_{\B_{r}} \V\left(\frac{|u-(u)_{r}|}{r}\right)\, dx \right)^{\frac{1}{p_{0}}},
\end{align*}
so from the definition of $\Psi_{\alpha}$ we infer
$$
{\mathcal H}(r)\leq W(r)+r+ \omega(r+{\tilde c}\,r^\alpha\Psi_\alpha(r)^{\frac{1}{p_0}}).
$$
\noindent Therefore, recalling that $r\leq r_{*}$ and $\tau \leq 1$, and by the inductive hypothesis on $\Psi_\alpha$, we have  
\begin{align*}
\h(\tau^{j} r) &\le{\rm W}(\tau^{j} r) +\tau^{j} r+ \omega\left( \tau^{j} r+ {\tilde c}(\tau^{j}r)^{\alpha} \Psi_{\alpha}(\tau^{j}r)^{\frac{1}{p_{0}}} \right)\\
&\leq {\rm W}(r_{*}) +r_*+ \omega\left( r_{*}+ {\tilde c}r_{*}^{\alpha} \kappa_*^{\frac{1}{p_{0}}} \right).
\end{align*}
\noindent Using the previous choice of constants and Lemma \ref{excessdecay}, we thus have
\begin{align}\label{eq27}
\Phi(\tau^{j+1} r) \leq c_1 \tau^{2} \Phi^{*}(\tau^{j}r) \leq 2c_1 \tau^{2} \mu_{*} \V(|D\ell_{\tau^{j}r}|). 
\end{align}
Let us consider
\begin{align*}
|D\ell_{\tau^{j}r}-D\ell_{\tau^{j+1}r} |&\leq \frac{n+2}{\tau^{j+1}r}\-int_{\B_{\tau^{j+1}r}} |u-\ell_{\tau^{j}r}|\, dx\\
&= \frac{n+2}{\tau}\frac{1}{(\tau^{j+1}r)^{n}\omega_{n}} \int_{\B_{\tau^{j+1}r}} \frac{|u-\ell_{\tau^{j}r}|}{\tau^{j}r}\, dx\\
&\leq \frac{n+2}{\tau^{n+1}} \-int_{\B_{\tau^{j}r}} \frac{|u-\ell_{\tau^{j}r}|}{\tau^{j}r}\, dx \\
&= \frac{n+2}{\tau^{n+1}} (\V_{|D\ell_{\tau^{j}r}|})^{-1} \circ \V_{|D\ell_{\tau^{j}r}|} \left(\-int_{\B_{\tau^{j}r}} \frac{|u-\ell_{\tau^{j}r}|}{\tau^{j}r}\, dx\right)\\
&\leq \frac{n+2}{\tau^{n+1}} (\V_{|D\ell_{\tau^{j}r}|})^{-1} \left(\-int_{\B_{\tau^{j}r}} \V_{|D\ell_{\tau^{j}r}|}\left(\frac{|u-\ell_{\tau^{j}r}|}{\tau^{j}r}\right)\, dx\right)\\
&=\frac{n+2}{\tau^{n+1}} (\V_{|D\ell_{\tau^{j}r}|})^{-1} \left( \Phi(\tau^{j}r)\right) \\
&\leq \frac{n+2}{\tau^{n+1}}( \V_{|D\ell_{\tau^{j}r}|})^{-1} \left( \mu_{*}\V(|D\ell_{\tau^{j}r}|)\right)\\
&\leq \frac{n+2}{\tau^{n+1}} \mu_{*}^{\frac{1}{p_{1}}}\V_{|D\ell_{\tau^{j}r}|}^{-1} \left(\V(|D\ell_{\tau^{j}r}|)\right)\le\frac{n+2}{\tau^{n+1}} \mu_{*}^{\frac{1}{p_{1}}}(\V_{|D\ell_{\tau^{j}r}|})^{-1}
 \left(4\V_{|D\ell_{\tau^{j}r}|}(|D\ell_{\tau^{j}r}|)\right)\\
&\le 4\frac{n+2}{\tau^{n+1}} \mu_{*}^{\frac{1}{p_{1}}}|D\ell_{\tau^{j}r}|\leq \frac{1}{2}|D\ell_{\tau^{j}r}|
\end{align*}
where we used \eqref{assitera} and the choice of $\mu_{*}$. 
This implies that
\begin{align*}
|D\ell_{\tau^{j}r}|\leq |D\ell_{\tau^{j}r}-D\ell_{\tau^{j+1}r}|+ |D\ell_{\tau^{j+1}r}| \leq \frac{1}{2}|D\ell_{\tau^{j}r}|+ |D\ell_{\tau^{j+1}r}|
\end{align*}
and so
\begin{align}\label{eq28}
|D\ell_{\tau^{j}r}| \leq 2 |D\ell_{\tau^{j+1}r}|.
\end{align}
Therefore, combining \eqref{eq27} and \eqref{eq28}, and taking into account \eqref{t1} and the choice of $\tau$, we have
\begin{align*}
\Phi(\tau^{j+1}r)&\leq 2 c_1\tau^{2}\mu_{*} \V(|D\ell_{\tau^{j}r}|) \leq 2 c_1\tau^{2}\mu_{*} \V(2|D\ell_{\tau^{j+1}r}|) \\
&\leq c_1\,\tau^{2}\mu_{*} 2^{p_{1}+1}\V(|D\ell_{\tau^{j+1}r}|) \leq \mu_{*}\V(|D\ell_{\tau^{j+1}r}|).  
\end{align*} 
Now, using $\ell_{\tau^{j}r}(x)= (u)_{\tau^{j}r} + D\ell_{\tau^{j}r}x$, \eqref{ok2.7}, and \eqref{t6}, we can see that
\begin{align*}
\Psi_{\alpha}(\tau^{j+1}r)&= (\tau^{j+1}r)^{p_{0}(1-\alpha)} \-int_{\B_{\tau^{j+1}r}} \V \left( \frac{|u- (u)_{\tau^{j+1}r}|}{\tau^{j+1}r} \right) \, dx \\
&\leq (\tau^{j+1}r)^{p_{0}(1-\alpha)}{c} \,\-int_{\B_{\tau^{j+1}r}} \V \left( \frac{|u- (u)_{\tau^{j}r}|}{\tau^{j+1}r} \right) \, dx \\
&\leq (\tau^{j+1}r)^{p_{0}(1-\alpha)}{c} \,\-int_{\B_{\tau^{j+1}r}} \V \left( \frac{|u- \ell_{\tau^{j}r}|}{\tau^{j+1}r} + \frac{|\ell_{\tau^{j}r}- (u)_{\tau^{j}r}|}{\tau^{j+1}r} \right) \, dx \\
&\leq (\tau^{j+1}r)^{p_{0}(1-\alpha)}{c} \, \-int_{\B_{\tau^{j+1}r}} \V \left( \frac{|u- \ell_{\tau^{j}r}|}{\tau^{j+1}r} +\left| D\ell_{\tau^{j}r}\right| \right) \, dx \\
&\leq (\tau^{j+1}r)^{p_{0}(1-\alpha)}{c} \, \-int_{\B_{\tau^{j+1}r}} \V_{|D\ell_{\tau^{j}r}|} \left(\frac{|u- \ell_{\tau^{j}r}|}{\tau^{j+1}r} \right) \, dx + (\tau^{j+1}r)^{p_{0}(1-\alpha)}{c} \V(|D\ell_{\tau^{j}r}|)\\
&\leq \frac{(\tau^{j+1}r)^{p_{0}(1-\alpha)}{c}}{\tau^{p_{1}+n}} \-int_{\B_{\tau^{j}r}}  \V_{|D\ell_{\tau^{j}r}|} \left(\frac{|u- \ell_{\tau^{j}r}|}{\tau^{j}r} \right) \, dx+ (\tau^{j+1}r)^{p_{0}(1-\alpha)}{c} \V(|D\ell_{\tau^{j}r}|)\\
&=\frac{(\tau^{j+1}r)^{p_{0}(1-\alpha)}{c}}{\tau^{p_{1}+n}} \Phi(\tau^{j}r) + (\tau^{j+1}r)^{p_{0}(1-\alpha)}{c} \V(|D\ell_{\tau^{j}r}|)\\
&\leq \frac{(\tau^{j+1}r)^{p_{0}(1-\alpha)}{c}}{\tau^{p_{1}+n}} \mu_{*} \V(|D\ell_{\tau^{j}r}|) + (\tau^{j+1}r)^{p_{0}(1-\alpha)}{c} \V(|D\ell_{\tau^{j}r}|).  
\end{align*}
Let us note that from 
\begin{align*}
|D\ell_{\tau^{j}r}| \leq \frac{n+2}{\tau^{j}r} \-int_{\B_{\tau^{j}r}} |u-(u)_{\tau^{j}r}| \, dx 
\end{align*}
we get
\begin{align*}
\V(|D\ell_{\tau^{j}r}|) &\leq (n+2)^{p_{1}} \-int_{\B_{\tau^{j}r}} \V\left(\frac{|u-(u)_{\tau^{j}r}|}{\tau^{j}r}\right) \, dx \\
&= (n+2)^{p_{1}} \frac{1}{(\tau^{j}r)^{p_{0}(1-\alpha)}} \Psi_{\alpha}(\tau^{j}r)\leq \frac{(n+2)^{p_{1}}}{(\tau^{j}r)^{p_{0}(1-\alpha)}}\kappa_{*}.
\end{align*}
Hence, by the choice of the constants, we finally get
\begin{equation}\label{cbar}
\begin{split}
\Psi_{\alpha}(\tau^{j+1}r) &\leq \frac{(\tau^{j+1}r)^{p_{0}(1-\alpha)}\bar{c}}{\tau^{p_{1}+n}} \mu_{*} \frac{(n+2)^{p_{1}}}{(\tau^{j}r)^{p_{0}(1-\alpha)}}\kappa_{*} + (\tau^{j+1}r)^{p_{0}(1-\alpha)}\bar{c} \frac{(n+2)^{p_{1}}}{(\tau^{j}r)^{p_{0}(1-\alpha)}}\kappa_{*}\\
&=\bar{c} \tau^{p_{0}(1-\alpha)-p_{1}-n}(n+2)^{p_{1}} \mu_{*}\kappa_{*} +\bar{c}\tau^{p_{0}(1-\alpha)}(n+2)^{p_{1}}\kappa_{*}\leq \kappa_{*}. 
\end{split}
\end{equation}
\end{proof}

\noindent We now establish an excess improvement estimate for the degenerate case. This will be achieved
via the $\V$-harmonic approximation lemma. We will be able to approximate our solution by a $\V$-harmonic function, allowing our solution to inherit the a priori estimates for $\V$-harmonic functions due to \cite{DSV}.

\begin{lem}\label{iterazioneDC} 
Let $\alpha, \mu, \kappa \in (0, 1)$; then, there exist constants $\sigma \in (0,1)$, $r_{\sharp}$ and $\delta_{*}$ such that if 
\begin{align}\label{assiterDC}
\mu \V(|D\ell_{x_{0}, r}|) \leq \Phi(x_{0}, r) \quad \mbox{ and }\quad  \Phi(x_{0}, r)\leq \delta_{*} 
\end{align}
with $r<r_{\sharp}$, then 
\begin{align*}
\Phi(x_{0}, \sigma r)\leq \delta_{*} \quad \mbox{ and } \quad \Psi_{\alpha}(x_{0}, \sigma r)<\kappa. 
\end{align*}
\end{lem} 

\begin{proof}
Assume $x_{0}=0$ and let us consider the function $h$, unique solution to 
\begin{align*}
\left\{
\begin{array}{ll}
-\dive \left( \V'(|Dh|) \frac{Dh}{|Dh|}\right) =0 &\mbox{ in } \B_{\frac{r}{4}}, \\
h=u &\mbox{ on } \partial \B_{\frac{r}{4}}. 
\end{array}
\right. 
\end{align*}
Let $\sigma \in \left(0, \frac{1}{4}\right)$ be such that 
\begin{align}\label{sigma}
\bar{c} \sigma^{2\alpha_{0}} \left(1+\frac{1}{\mu}\right)<1\ \ \ \ \ \hbox{ and  } \ \ \ \ \ \sigma^{p_{0}(1-\alpha)} <\frac{\mu}{\mu+1} \frac{1}{\tilde{c}} \kappa,
\end{align}
where $\bar c$ and $\tilde c$ will be defined later.
We set $\e=\sigma^{2\alpha_{0}+n}$ and $\delta_0$ be the corresponding constant from Theorem \ref{phiappr}, with $\alpha_0\in (0,1)$ defined by \eqref{decayh} (relative to $h$); moreover we choose $c_{1}\delta=\frac{1}{3}\delta_{0}$, where $c_1$ and $\delta$ come from Lemma \ref{lem1DC}.

Then, using \eqref{t6} and \eqref{assiterDC}, it yields
\begin{align}\label{DC1}
\Psi_{1}\left(\frac{r}{2}\right)&=\-int_{\B_{\frac{r}{2}}}\V\left(\frac{|u-(u)_{\frac{r}{2}}|}{\frac{r}{2}}\right)\, dx \leq c\-int_{\B_{r}}\V \left( \frac{|u-(u)_{r}|}{r}\right)\, dx \nonumber \\
&\leq c\-int_{\B_{r}} \V\left( \frac{|u-\ell_{r}|}{r}+|D\ell_{r}|\right)\, dx \nonumber \\
&\leq \tilde{c} \left[ \-int_{\B_{r}} \V_{|D\ell_{r}|}\left( \frac{|u-\ell_{r}|}{r}\right)\, dx+ \V(|D\ell_{r}|)\right]  \nonumber \\
&\leq \tilde{c}\, \Phi(r)+ \frac{\tilde{c}}{\mu} \Phi(r) \leq \delta_{*}\tilde{c}\left(1+\frac{1}{\mu}\right),
 \end{align}
hence
\begin{equation}\label{deltastar}
\frac{c_{1}}{\eta(\delta)} \V^{-1}\left(\Psi_{1}\left(\frac{r}{2}\right)\right) \leq \frac{c_{1}}{\eta(\delta)}  \V^{-1}\left( \delta_{*}\tilde{c}\left(1+\frac{1}{\mu}\right)\right)\leq \frac{1}{3}\delta_{0}(\e),
\end{equation}
provided $\delta_*$ is chosen sufficiently small to guarantee \eqref{deltastar}.
We note that from \eqref{DC1}, $\delta_{*}<1$, and $r<r_{\sharp}$ we have also
\begin{align*}
\-int_{\B_{r}} |u-(u)_{r}|\, dx &= r\-int_{\B_{r}} \frac{|u-(u)_{r}|}{r}\, dx = r \V^{-1} \circ \V \left( \-int_{\B_{r}} \frac{|u-(u)_{r}|}{r}\, dx\right) \\
&\leq r \V^{-1} \left( \Psi_{1}(r)\right) \leq r \V^{-1}\left( \delta_{*}\tilde{c}\left(1+\frac{1}{\mu}\right)\right)\\
&\leq r \V^{-1} \left(\tilde{c}\left(1+\frac{1}{\mu}\right)\right) \leq r_{\sharp} c_{*},
\end{align*}
where $c_{*}=\V^{-1} \left(\tilde{c}\left(1+\frac{1}{\mu}\right)\right)$. Now, we choose $r_{\sharp}$ such that 
\begin{align*}
c_{1} \left[ {\rm W}(r_{\sharp})^{\frac{\gamma-1}{\gamma}} + \omega (c_{*}r_{\sharp})^{\frac{\gamma-1}{\gamma}}+r_{\sharp}\right]\leq \frac{1}{3}\delta_{0}. 
\end{align*}
Therefore, from Lemma \ref{lem1DC} we have
\begin{align*}
\left| \-int_{\B_{\frac{r}{4}}} \left( \V'(|Du|)\frac{Du}{|Du|}, D\zeta \right) \, dx \right| \leq \delta_{0} \left[ \-int_{\B_{\frac{r}{2}}} \V(|Du|)\, dx + \V(\|D\zeta\|_{\infty})\right]. 
\end{align*}
Hence we can apply Theorem \ref{phiappr} to assert that the function $h$ satisfies 
\begin{align*}
\left(\-int_{\B_{\frac{r}{4}}} |V(Du)- V(Dh)|^{2\theta} \, dx \right)^{\frac{1}{\theta}} \leq \e \-int_{\B_{\frac{r}{2}}} \V(|Du|)\, dx=\sigma^{2\alpha_{0}+n} \-int_{\B_{\frac{r}{2}}} \V(|Du|)\, dx,  
\end{align*}
recalling the definition of $\e$.
From \cite[Corollary 2.10]{CO} it follows that
\begin{align}\label{DC2}
\-int_{\B_{\frac{r}{4}}} |V(Du)- V(Dh)|^{2} \, dx \leq c\,\sigma^{2\alpha_{0}+n} \-int_{\B_{\frac{r}{2}}} \V(|Du|)\, dx.  
\end{align}
Now, let $\sigma <\frac{1}{4}$. Then, from Poincar\`e's inequality
\begin{align*}
\Phi(\sigma r)&=\-int_{\B_{\sigma r}} \V_{|D\ell_{\sigma r}|} \left( \frac{|u-\ell_{\sigma r}|}{\sigma r}\right)\, dx \\
&\leq c_p \-int_{\B_{\sigma r}} \V_{|D\ell_{\sigma r}|} \left( |Du-D\ell_{\sigma r}|\right)\, dx\\
&\leq c \-int_{\B_{\sigma r}} |V(Du)-V(D\ell_{\sigma r})|^{2}\, dx\\
&\leq c\left[ \-int_{\B_{\sigma r}} |V(Du)-V(Dh)|^{2}\, dx + \-int_{\B_{\sigma r}} |V(Dh)-V((Dh)_{\sigma r})|^{2}\, dx \right. \\
&\qquad \left.+ \-int_{\B_{\sigma r}} |V((Dh)_{\sigma r})- V(D\ell_{\sigma r})|^{2}\, dx \right]=I+I\!I + I\!I\!I. 
\end{align*}
{\it Estimate of $I$.} Using \eqref{DC2}
\begin{align*}
I\leq \frac{c}{\sigma^{n}} \-int_{\B_{\frac{r}{4}}} |V(Du)-V(Dh)|^{2}\, dx \leq c \sigma^{2\alpha_{0}} \-int_{\B_{\frac{r}{2}}} \V(|Du|)\, dx.  
\end{align*}
{\it Estimate of $I\!I$.} From \eqref{decayh} and using the fact that $|V(Dh)|^{2}\sim \V(|Dh|)$ we have 
\begin{align*}
I\!I&\leq c\, \sigma^{2\alpha_{0}} \-int_{\B_{\frac{r}{4}}} |V(Dh) - (V(Dh))_{\frac{r}{4}}|^{2}\, dx \leq c\, \sigma^{2\alpha_{0}} \-int_{\B_{\frac{r}{4}}} \V(|Dh|) \, dx \\
&\leq c\, \sigma^{2\alpha_{0}} \-int_{\B_{\frac{r}{4}}} \V(|Du|) \, dx\leq c\, \sigma^{2\alpha_{0}} \-int_{\B_{\frac{r}{2}}} \V(|Du|) \, dx.
\end{align*}
{\it Estimate of $I\!I\!I$.} From the definition of the $V$-function we have 
\begin{align*}
I\!I\!I\leq c\,\V_{|(Dh)_{\sigma r}|} \left( |D\ell_{\sigma r} - (Dh)_{\sigma r}|\right). 
\end{align*}
By \eqref{minimizza}, we note that 
\begin{align*}
|D\ell_{\sigma r} - (Dh)_{\sigma r}| &\leq \frac{n+2}{\sigma r} \-int_{\B_{\sigma r}} |u-(u)_{\sigma r} -(Dh)_{\sigma r} x|\, dx\leq (n+2) c\, \-int_{\B_{\sigma r}} |Du- (Dh)_{\sigma r}|\, dx,
\end{align*}
hence  it yields
\begin{align*}
I\!I\!I &\leq c \-int_{\B_{\sigma r}} \V_{|(Dh)_{\sigma r}|} \left( |Du - (Dh)_{\sigma r}|\right)\, dx \leq c \-int_{\B_{\sigma r}} |V(Du) - V((Dh)_{\sigma r})|^{2}\, dx \\
&\leq c\left[ \-int_{\B_{\sigma r}} |V(Du) - V(Dh)|^{2}\, dx  + \-int_{\B_{\sigma r}} |V(Dh) - V((Dh)_{\sigma r})|^{2}\, dx \right]\\
&\leq c\, \sigma^{2\alpha_{0}} \-int_{\B_{\frac{r}{2}}} \V(|Du|) \, dx. 
\end{align*}
Combining the previous estimates we obtain 
\begin{align*}
\Phi(\sigma r)\leq c\, \sigma^{2\alpha_{0}} \-int_{\B_{\frac{r}{2}}} \V(|Du|) \, dx. 
\end{align*}
Using Theorem \ref{cacczero}, \eqref{t6}, and \eqref{assiterDC} we get
\begin{align*}
\Phi(\sigma r)&\leq c\, \sigma^{2\alpha_{0}} \-int_{\B_{\frac{r}{2}}} \V(|Du|) \, dx\leq c\, \sigma^{2\alpha_{0}} \-int_{\B_{r}} \V \left( \frac{|u-(u)_{r}|}{r}\right) \, dx\leq c \,\sigma^{2\alpha_{0}} \left[ \-int_{\B_{r}} \V \left( \frac{|u-\ell_{r}|}{r}+|D\ell_{r}|\right)\right]\\
&\leq \bar{c}\, \sigma^{2\alpha_{0}} \left[ \-int_{\B_{r}} \V_{|D\ell_{r}|} \left( \frac{|u-\ell_{r}|}{r}\right) \, dx + \V(|D\ell_{r}|)\right]\leq \bar{c}\, \sigma^{2\alpha_{0}} \left(1+\frac{1}{\mu}\right) \Phi(r)\leq \bar{c} \,\sigma^{2\alpha_{0}} \left(1+\frac{1}{\mu}\right) \delta_{*}\\
&\le  \delta_{*}, 
\end{align*}
according to \eqref{sigma}.
Moreover, from \eqref{t6}, Poincar\'e's inequality, and \eqref{suph}, and proceeding as for the estimates of $I$, $I\!I$, and $I\!I\!I$ we get  
\begin{align*}
\Psi_{\alpha}(\sigma r) &= (\sigma r)^{p_{0}(1-\alpha)} \-int_{\B_{\sigma r}} \V \left( \frac{|u-(u)_{\sigma r}|}{\sigma r}\right)\, dx \\
&\leq (\sigma r)^{p_{0}(1-\alpha)} \-int_{\B_{\sigma r}} \V \left( \frac{|u-(u)_{\sigma r}-(Dh)_{{\sigma r}}x|}{\sigma r} + |(Dh)_{{\sigma r}}|\right)\, dx\\
&\leq c (\sigma r)^{p_{0}(1-\alpha)} \left[\-int_{\B_{\sigma r}} \V_{|(Dh)_{{\sigma r}}|} \left( \frac{|u-(u)_{\sigma r}-(Dh)_{{\sigma r}}x|}{\sigma r}\right)\, dx  + \V\left(|(Dh)_{{\sigma r}}|\right) \right]\\
&\leq c (\sigma r)^{p_{0}(1-\alpha)} \left[\-int_{\B_{\sigma r}} \V_{|(Dh)_{{\sigma r}}|} \left( |Du-(Dh)_{{\sigma r}}|\right)\, dx  + \-int_{\B_{{\sigma r}}} \V(|Dh|)\, dx  \right]\\
&\leq c (\sigma r)^{p_{0}(1-\alpha)} \left[\-int_{\B_{\sigma r}} |V(Du)-V((Dh)_{{\sigma r}})|^{2} \, dx  + \sup_{\B_{{\sigma r}{} }} \V(|Dh|)\, \right]\\
&\leq c (\sigma r)^{p_{0}(1-\alpha)}\! \left[\-int_{\B_{\sigma r}} \!|V(Du)-V(Dh)|^{2} \, dx +\-int_{\B_{\sigma r}} \!|V(Dh)-V((Dh)_{{\sigma r}})|^{2} \, dx +  \-int_{\B_{\frac{r}{2}}} \! \V(|Dh|)\, dx \right]\\
&\leq c (\sigma r)^{p_{0}(1-\alpha)} \sigma^{2\alpha_{0}} \-int_{\B_{\frac{r}{2}}} \V(|Du|)\, dx + c (\sigma r)^{p_{0}(1-\alpha)} \-int_{\B_{\frac{r}{2}}} \V(|Du|)\, dx\\
&\leq c (\sigma r)^{p_{0}(1-\alpha)} \-int_{\B_{\frac{r}{2}}} \V(|Du|)\, dx\leq \tilde{c} (\sigma r)^{p_{0}(1-\alpha)} \left(1+\frac{1}{\mu}\right) \delta_{*}\leq \tilde{c}\sigma^{p_{0}(1-\alpha)} \left(1+\frac{1}{\mu}\right) \leq \kappa,
\end{align*}
according to \eqref{sigma}. This completes the proof of the lemma.
\end{proof}

\noindent In the following lemma we combine the degenerate and the non-degenerate regime.

\begin{lem}\label{lemconclusivo}
Let $\alpha \in (0, 1)$. Then there exist constants $\delta_{*}, \kappa_{*}\in (0, 1)$ and $r_{1}, C$ such that the conditions
\begin{align}\label{DC3}
\Phi(x_{0}, r)<\delta_{*} \quad \mbox{ and } \quad \Psi_{\alpha}(x_{0}, r)<\kappa_{*} 
\end{align}
for $\B_{r}(x_{0}) \subset \Omega$ with $r\leq r_{1}$, imply 
\begin{align*}
\sup_{\rho \in (0, r]} \Psi_{\alpha}(x_{0}, \rho)\leq C. 
\end{align*}
\end{lem}

\begin{proof}
Assume $x_{0}=0$. Let $\mu_{*}\in (0,1)$, $\kappa_{*}\in (0,1)$, $r_{*}\in (0, \rho_{0})$ and $\tau \in \left(0, \frac{1}{8}\right)$ be given by Lemma \ref{iterazione}. Let us choose $\mu= \mu_{*}$ and $\kappa= \kappa_{*}$ in Lemma \ref{iterazioneDC}. Set $r_{1}= \min \{r_{\sharp}, r_{*}\}$, and take $r<r_{1}$. 

Define
\begin{align*}
{\mathcal S}= \left\{ k\in \N_{0} \, : \, \mu_{*} \V(|D\ell_{\sigma^{k}r}|)\leq \Phi(\sigma^{k}r) \right\}. 
\end{align*}
We distinguish two cases.  

{\bf First case: ${\mathcal S}= \N_{0}$.} We claim that 
\begin{align}\label{DC4}
\Phi(\sigma^{k}r)<\delta_{*} \quad \mbox{ and }\quad  \Psi_{\alpha}(\sigma^{k}r)<\kappa_{*} \quad  \forall k\in \N_{0}. 
\end{align}
We prove the claim by induction on $k$. If $k=0$ then \eqref{DC4} follows by \eqref{DC3}. Assume that \eqref{DC4} is verified by $k\in \N_{0}$ and let us prove it for $k+1$. By applying Lemma \ref{iterazioneDC} to $\sigma^{k}r$ in place of $r$ we get the thesis. Now we prove that 
\begin{align}\label{DC5}
\sup_{\rho \in (0, r]} \Psi_{\alpha}(\rho)\leq C. 
\end{align}
Let $\rho \in (0, r]$. Then there exists $k\in \N_{0}$ such that 
\begin{align*}
\sigma^{k+1}r <\rho \leq \sigma^{k}r. 
\end{align*}
Then, by Remark \ref{mediamin}, \eqref{t1} and \eqref{DC4} we infer
\begin{align*}
\Psi_{\alpha}(\rho)&= \rho^{p_{0}(1-\alpha)} \-int_{\B_{\rho}} \V \left( \frac{|u-(u)_{\rho}|}{\rho}\right)\, dx \\
&\leq \frac{c(\sigma^{k}r)^{p_{0}(1-\alpha)}}{\sigma^{n} |\B_{\sigma^{k}r}|} \int_{\B_{\rho}} \V \left( \frac{|u-(u)_{\sigma^{k}r}|}{\rho}\right)\, dx \\
&\leq \frac{c(\sigma^{k}r)^{p_{0}(1-\alpha)}}{\sigma^{n}} \-int_{\B_{\sigma^{k}r}} \V \left( \frac{|u-(u)_{\sigma^{k}r}|}{\sigma^{k+1}r}\right)\, dx \\
&\leq \frac{c(\sigma^{k}r)^{p_{0}(1-\alpha)}}{\sigma^{n+p_{1}}} \-int_{\B_{\sigma^{k}r}} \V \left( \frac{|u-(u)_{\sigma^{k}r}|}{\sigma^{k}r}\right)\, dx \\
&= \frac{c}{\sigma^{n+p_{1}}} \Psi_{\alpha}(\sigma^{k}r)\leq \frac{c}{\sigma^{n+p_{1}}} \kappa_{*}. 
\end{align*}
Hence \eqref{DC5} holds true. 

{\bf Second case: ${\mathcal S}\neq \N_{0}$}.
Then, there exists $k_{0}= \min \{\N_{0}\setminus {\mathcal S}\} \in \N_{0}$ such that for any $k=0, \dots, k_{0}-1$ (if $k_{0}\geq 1$) 
\begin{align*}
\mu_{*}\V(|D\ell_{\sigma^{k}r}|)\leq \Phi(\sigma^{k}r) \quad \mbox{ and } \quad \mu_{*}\V(|D\ell_{\sigma^{k_{0}}r}|)> \Phi(\sigma^{k_{0}}r). 
\end{align*}
Proceeding as in the first case, we deduce that for any $k\leq k_{0}$ 
\begin{align}\label{DC6}
\Phi(\sigma^{k}r)<\delta_{*} \quad \mbox{ and }\quad  \Psi_{\alpha}(\sigma^{k}r)<\kappa_{*}. 
\end{align}
Therefore 
\begin{align*}
\frac{\Phi(\sigma^{k_{0}}r)}{\V(|D\ell_{\sigma^{k_{0}}r}|)}<\mu_{*} \quad \mbox{ and } \quad \Psi_{\alpha}(\sigma^{k_{0}}r)<\kappa_{*}. 
\end{align*}
Thus we are in the position to apply Lemma \ref{iterazione} to deduce that
\begin{align*}
\frac{\Phi(\tau^{j}\sigma^{k_{0}}r)}{\V(|D\ell_{\tau^{j}\sigma^{k_{0}}r}|)}<\mu_{*} \quad \mbox{ and } \quad \Psi_{\alpha}(\tau^{j}\sigma^{k_{0}}r)<\kappa_{*} \quad \forall j\in \N_{0}. 
\end{align*}
Let $\rho \in (0, r]$. If $\rho \in \left(\frac{\sigma^{k_{0}}r}{2}, r \right]$ then there exists $k\leq k_{0}$ such that $\sigma^{k+1}r <\rho \leq \sigma^{k}r$. Hence we are in the same position of the first case, and we deduce that 
\begin{align*}
\Psi_{\alpha}(\rho)\leq \frac{c}{\sigma^{n+p_{1}}} \kappa_{*}. 
\end{align*}
Now, if $\rho \in \left(0, \frac{\sigma^{k_{0}}r}{2}\right]$, then there exists $j\in \N_{0}$ such that $\tau^{j+1}\sigma^{k_{0}}r <\rho \leq \tau^{j}\sigma^{k_{0}}r$. Therefore
\begin{align*}
\Psi_{\alpha}(\rho) &=\rho^{p_{0}(1-\alpha)} \-int_{\B_{\rho}} \V \left( \frac{|u-(u)_{\rho}|}{\rho}\right)\, dx \\
&\leq c \frac{(\tau^{j}\sigma^{k_{0}}r)^{p_{0}(1-\alpha)}}{\tau^{n} |\B_{\tau^{j}\sigma^{k_{0}}r}|} \int_{\B_{\rho}} \V \left( \frac{|u-(u)_{\tau^{j}\sigma^{k_{0}}r}|}{\rho}\right)\, dx \\
&\leq \frac{c}{\tau^{n+p_{1}}} \Psi_\alpha(\tau^{j}\sigma^{k_{0}}r) < \frac{c}{\tau^{n+p_{1}}} \kappa_{*}. 
\end{align*}
Then we conclude that 
\begin{align*}
\sup_{\rho \in (0, r]} \Psi_{\alpha}(\rho) \leq c\,\kappa_{*} \max \left\{ \frac{1}{\tau^{n+p_{1}}}, \frac{1}{\sigma^{n+p_{1}}} \right\},
\end{align*}
and this completes the proof of the lemma.
\end{proof}

\section{Proof of the main result}

\noindent
Let $x_{0}$ be such that 
\begin{align*}
&\liminf_{r\ri 0} \-int_{\B_{r}(x_{0})} |V(Du)- (V(Du))_{x_{0}, r}|^{2}\, dx =0,\\
&\limsup_{r\ri 0} \left|\-int_{\B_{r}(x_{0})} Du\, dx \right|<\infty. 
\end{align*}

\noindent
Using Lemma 22 in \cite{DLSV} we have that
\begin{align*}
\liminf_{r\ri 0} \-int_{\B_{r}(x_{0})} \V_{|(Du)_{x_{0}, r}|}(|Du- (Du)_{x_{0}, r}|)\, dx=0.  
\end{align*}

\noindent
Exploiting Poincar\'e's inequality together with \eqref{t6} we infer
\begin{align*}
\Psi_{\alpha}(x_{0}, r)&= r^{p_{0}(1-\alpha)} \-int_{\B_{r}(x_{0})} \V \left( \left| \frac{u-(u)_{x_{0}, r}}{r}\right|\right)\, dx \\
&\leq c\, r^{p_{0}(1-\alpha)} \-int_{\B_{r}(x_{0})} \V(|Du|)\, dx \\
&\leq c\, r^{p_{0}(1-\alpha)} \-int_{\B_{r}(x_{0})} \V(|Du - (Du)_{x_{0}, r}| + |(Du)_{x_{0}, r}|)\, dx \\
&\leq c\, r^{p_{0}(1-\alpha)} \-int_{\B_{r}(x_{0})} \V_{|(Du)_{x_{0}, r}|}(|Du - (Du)_{x_{0}, r}|)\, dx + c\,r^{p_{0}(1-\alpha)} \V(|(Du)_{x_{0}, r}|). 
\end{align*}
Furthermore, using Lemma \ref{shift}, we get
\begin{align*}
\Phi(x_{0}, r) &=\-int_{\B_{r}} \V_{|(D\ell)_{x_{0}, r}|} \left( \frac{|u-\ell_{x_{0}, r}|}{r}\right) \, dx \\
&\leq c_{P} \-int_{\B_{r}} \V_{|(D\ell)_{x_{0}, r}|} \left( |Du-(Du)_{x_{0}, r}|\right) \, dx \\
&\leq c \-int_{\B_{r}} \V_{|(Du)_{x_{0}, r}|} \left( |Du-(Du)_{x_{0}, r}|\right) \, dx + c \V_{|(D\ell)_{x_{0}, r}|} \left( |(Du)_{x_{0}, r}-(D\ell)_{x_{0}, r}|\right) \\
&\leq c \-int_{\B_{r}} \V_{|(Du)_{x_{0}, r}|} \left( |Du-(Du)_{x_{0}, r}|\right) \, dx + c \V_{|(Du)_{x_{0}, r}|} \left( |(Du)_{x_{0}, r}-D\ell_{x_{0}, r}|\right) \\
&\leq c \-int_{\B_{r}} \V_{|(Du)_{x_{0}, r}|} \left( |Du-(Du)_{x_{0}, r}|\right) \, dx.  
\end{align*}
Keeping in mind the choice of $x_0$, these estimates lead to the existence of a radius $\rho<r_{1}$ such that 
\begin{align*}
\Phi(x_{0}, \rho)<\delta_{*} \quad \mbox{ and } \quad \Psi_{\alpha}(x_{0}, \rho)<\kappa_{*},
\end{align*}
where $r_1$, $\delta_*$, and $\kappa_*$ are given in Lemma \ref{lemconclusivo}.
The absolute continuity of the integral entails the existence of a neighborhood $U$ of $x_{0}$ such that:
\begin{align*}
\Phi(x, \rho)<\delta_{*} \quad \mbox{ and } \quad \Psi_{\alpha}(x, \rho)<\kappa_{*} \quad \forall x\in U.  
\end{align*}
So, by Lemma \ref{lemconclusivo} we get
\begin{align*}
\sup_{x\in U, \sigma\in (0, \rho]} \Psi_{\alpha}(x, \sigma)<\infty,
\end{align*}
from which we finally deduce
\begin{align*}
&\sigma^{p_{0}(1-\alpha)} \-int_{\B_{\sigma}(x)} \left| \frac{u-(u)_{x, \sigma}}{\sigma}\right|^{p_{0}}\, dx \\
&= \frac{\sigma^{p_{0}(1-\alpha)}}{|\B_{\sigma}|} \int_{\B_{\sigma}\cap \left\{\left| \frac{u-(u)_{x, \sigma}}{\sigma}\right|\leq 1 \right\} } \left| \frac{u-(u)_{x, \sigma}}{\sigma}\right|^{p_{0}}\, dx + \frac{\sigma^{p_{0}(1-\alpha)}}{|\B_{\sigma}|\V(1)} \int_{\B_{\sigma}\cap \left\{\left| \frac{u-(u)_{x, \sigma}}{\sigma}\right|> 1 \right\} } \left| \frac{u-(u)_{x, \sigma}}{\sigma}\right|^{p_{0}} \V(1)\, dx \\
&\leq \sigma^{p_{0}(1-\alpha)} + \frac{\sigma^{p_{0}(1-\alpha)}}{|\B_{\sigma}|\V(1)} \int_{\B_{\sigma}} \V\left(\left| \frac{u-(u)_{x, \sigma}}{\sigma}\right|\right) \, dx,
\end{align*}
and 
\begin{align*}
\sup_{x\in U, \sigma \in (0, \rho)} \frac{1}{\sigma^{n+p_{0}\alpha}} \int_{\B_{\sigma}(x)} |u-(u)_{x, \sigma}|^{p_{0}}\, dx <\infty.
\end{align*}
Using Campanato's characterization of H\"older continuous functions we eventually conclude that $u\in \C^{0, \alpha}(U, \R^{N})$.


\addcontentsline{toc}{section}{\refname}
\begin{thebibliography}{200}





\bibitem{BJDE}
V. B\"ogelein, 
{\it Partial regularity for minimizers of discontinuous quasi--convex integrals with degeneracy}, 
J. Differential Equations {\bf 252} (2012), no. 2, 1052--1100. 

\bibitem{BDHS}
V. B\"ogelein, F. Duzaar, J. Habermann and C. Scheven, 
{\it Partial H\"older continuity for discontinuous elliptic problems with VMO--coefficients}, 
 Proc. Lond. Math. Soc. (3) {\bf 103} (2011), no. 3, 371--404. 


\bibitem{BV}
D. Breit and A. Verde, 
{\it Quasiconvex variational functionals in Orlicz--Sobolev spaces},  
Ann. Mat. Pura Appl. (4) {\bf 192} (2013), no. 2, 255--271. 


\bibitem{CO}
P. Celada and J. Ok, 
{\it Partial regularity for non--autonomous degenerate quasi--convex functionals with general growth},
Nonlinear Anal. {\bf 194} (2020), 111473, 36 pp. 

\bibitem{degiorgi} 
E. De Giorgi, 
{\it Frontiere orientate di misura minima}, 
Seminario di Matematica della Scuola Normale Superiore di Pisa 1960--61, Editrice Tecnico Scienti ca, Pisa, 1961, p. 57.

\bibitem{DE}
L. Diening and F. Ettwein, 
{\it Fractional estimates for non-differentiable elliptic systems with general growth},
Forum Math. {\bf 20} (2008), 523--556.

\bibitem{DKr}
L. Diening, C. Kreuzer, {\it Linear convergence of an adaptive finite element method for the $p$-Laplacian equation}, SIAM J. Num. Anal. {\bf 46} (2) (2008), 614--638. 

\bibitem{DLSV}
L. Diening, D. Lengeler, B. Stroffolini and  A. Verde, 
{\it Partial regularity for minimizer of quasi-convex functional with general growth}, 
SIAM J. Math. Anal. {\bf 44} (5) (2012), 3594--3616. 

\bibitem{DSV}
L. Diening, B. Stroffolini and  A. Verde, 
{\it Everywhere regularity of functionals with $\phi$-growth}, 
Manuscripta Math. {\bf 129} (2009), no. 4, 449--481.

\bibitem{DSV1}
L. Diening, B. Stroffolini and A. Verde, 
{\it The $\varphi$-harmonic approximation and the regularity of $\varphi$-harmonic maps}, 
J. Differential Equations {\bf 253} (2012), no. 7, 1943--1958.

\bibitem{DES}
P. Di Gironimo, L. Esposito and L. Sgambati,
{\it A remark on $L^{2, \lambda}$ regularity for minimizers of quadratic functionals}, 
Manuscripta Mathematica {\bf 113} (2004), 143--151.
 
\bibitem{DG}
F. Duzaar and J.F. Grotowski, 
{\it Optimal interior partial regularity for nonlinear elliptic systems: The method of $\A$-harmonic approximation}, 
Manuscripta Math. {\bf 103}, 2000, 267--298.

\bibitem{DGK}
F. Duzaar, J.F. Grotowski and M. Kronz, {\it Regularity of almost minimizers of quasi--convex variational integrals with subquadratic growth}, Ann. Mat. Pura Appl. (4) {\bf 184} (2005), no. 4, 421--448.

\bibitem{DK}
F. Duzaar and M. Kronz, 
{\it Regularity of $\omega$-minimizers of quasi-convex variational integrals with polynomial growth},  
8th International Conference on Differential Geometry and its Applications (Opava, 2001). 
Differential Geom. Appl. {\bf 17} (2002), no. 2--3, 139--152.

\bibitem{DM}
F. Duzaar and G. Mingione,
{\it The $p$-harmonic approximation and the regularity of $p$-harmonic maps}, 
Calc. Var. Partial Differential Equations {\bf 20}, 2004, 235--256.

\bibitem{DM1}
F. Duzaar and G. Mingione, 
{\it Harmonic type approximation lemmas}, 
J. Math. Anal. Appl., {\bf 352}, (1), 2009, 301--335.

\bibitem{FM}
M. Foss and G. Mingione, 
{\it Partial continuity for elliptic problems},  
Ann. Inst. H. Poincar\'e  Anal. Non Lin\'eaire {\bf 25} (2008), no. 3, 471--503.


\bibitem{GM}
E. Giusti and M. Miranda, 
{\it Sulla regolarit\`a delle soluzioni deboli di una classe di sistemi ellittici quasilineari}, 
Arch. Ration. Mech. Anal. {\bf 31} (1968/69) 173--184.

\bibitem{GSS}
C. Goodrich, G. Scilla and B. Stroffolini, 
{\it Partial H\"older continuity for minimizers of discontinuous quasiconvex integrals with VMO coefficients and general growth}, 
Preprint 2021. 

\bibitem{ILV}
T. Isernia, C. Leone and A. Verde, 
{\it Partial regularity results for asymptotic quasiconvex functionals with general growth}, 
 Ann. Acad. Sci. Fenn. Math. {\bf 41} (2016), no. 2, 817--844. 

\bibitem{KM1}
J. Kristensen and G. Mingione, 
{\it The singular set of minima of integral functionals}, 
Arch. Ration. Mech. Anal. {\bf 180} (3) (2006), 331--398.

\bibitem{KM2} 
J. Kristensen and G. Mingione, 
{\it The singular set of Lipschitzian minima of multiple integrals}, 
Arch. Ration. Mech. Anal. {\bf 184} (2007) 341--369.

\bibitem{KM3}
J. Kristensen and G. Mingione, 
{\it Boundary regularity in variational problems}, 
Arch. Ration. Mech. Anal. {\bf 198} (2010), no. 2, 369--455.

\bibitem{Lieberman}
G.M. Lieberman, 
{\it Boundary regularity for solutions of degenerate parabolic equations}, 
Nonlinear Anal. {\bf 14} (1990), no. 6, 501--524

\bibitem{M1}
P. Marcellini, 
{\it Regularity of minimizers of integrals of the calculus of variations with nonstandard growth conditions}, 
Arch. Rational Mech. Anal. {\bf 105} (1989), no. 3, 267--284.

\bibitem{M2}
P. Marcellini, 
{\it Regularity and existence of solutions of elliptic equations with $p,q$--growth conditions}, 
J. Differential Equations {\bf 90} (1991), no. 1, 1--30. 

\bibitem{M3} 
P. Marcellini, 
{\it Everywhere regularity for a class of elliptic systems without growth conditions}, 
Ann. Scuola Norm. Sup. Pisa Cl. Sci. (4) {\bf 23} (1996), no. 1, 1--25. 

\bibitem{Morrey}
C.B. Morrey, 
{\it Quasi--convexity and the lower semicontinuity of multiple integrals}, 
Pacific J. Math. {\bf 2} (1952), 25--53. 

\bibitem{ok}
J. Ok, 
{\it Partial H\"older regularity for elliptic systems with non-standard growth}, 
J. Funct. Anal. {\bf 274} (3) (2018), 723--768.


\bibitem{RT1}
M. Ragusa and A. Tachikawa, 
{\it Partial regularity of the minimizers of quadratic functionals with VMO coefficients}, 
J. Lond. Math. Soc. (2) {\bf 72} (2005), 609--620.

\bibitem{raoren} 
M. Rao and Z.D. Ren, 
{\it Theory of Orlicz spaces}, 
Monographs and Textbooks in Pure and Applied Mathematics 146 Marcel Dekker, Inc., New York (1991).

\bibitem{Simon}
L. Simon, 
{\it Theorems on regularity and singularity of energy minimizing maps}, 
Lectures Math. ETH Zürich, Birkh\"auser Verlag, Basel, 1996.


\bibitem{Str}
B. Stroffolini, {\it Partial Regularity Results for Quasimonotone Elliptic Systems with General Growth}, Z. Anal. Anwend. {\bf 39} (2020), 315-347.

\bibitem{Z}
S.Z. Zheng, 
{\it Partial regularity for quasi-linear elliptic systems with VMO coefficients under the natural growth}, 
Chinese Ann. Math. Ser. A  {\bf 29} (2008), 49--58.


\end{thebibliography}
\end{document}